\documentclass[12pt,reqno]{amsart}
\usepackage{graphicx} 
\usepackage{amsmath, amsfonts, amssymb, amsthm}
\usepackage{amsmath,amssymb,amscd,verbatim,color}
\usepackage[backref]{hyperref}
\usepackage{MnSymbol}
\usepackage{stmaryrd}
\usepackage{amsfonts}
\usepackage[top=35mm, bottom=35mm, left=30mm, right=30mm]{geometry}
\usepackage{xypic}
\usepackage{xcolor}
\usepackage{amsthm}
\usepackage{tikz}
\usepackage{tikz-cd}
\theoremstyle{plain}
\newtheorem{thm}{Theorem}[section]
\newtheorem{cor}[thm]{Corollary}
\newtheorem{lem}[thm]{Lemma}
\newtheorem{prop}[thm]{Proposition}
\newtheorem{conj}[thm]{Conjecture}

\theoremstyle{definition}
\newtheorem{defn}[thm]{Definition}

\newtheorem{rem}[thm]{Remark}

\newtheorem{notation}[thm]{Notation}

\numberwithin{equation}{section}

\newcommand{\cO}{\mathcal{O}}
\newcommand{\N}{\mathbb{N}}
\newcommand{\tf}{\Tilde{f}}
\newcommand{\udx}{\underline{x}}
\newcommand{\udy}{\underline{y}}
\newcommand{\Reg}{\operatorname{Reg}}

\newcommand{\supp}{\operatorname{supp}}
\newcommand{\Z}{\mathbb{Z}}

\providecommand{\abs}[1]{\lvert#1\rvert}

 \def \D{\Delta}

\def\R{\mathbb{R}}

\def\m{\mathcal{M}}
\def\B{\mathcal{B}}
\def\D{\mathcal{D}}

\def\lip{{\rm Lip}}

\allowdisplaybreaks

\renewcommand*{\backref}[1]{}

\makeatletter

\def\MRbibitem{\@ifnextchar[\my@lbibitem\my@bibitem}

\def\mybiblabel#1#2{\@biblabel{{\hyperref{http://www.ams.org/mathscinet-getitem?mr=#1}{}{}{#2}}}}

\def\myhyperanchor#1{\Hy@raisedlink{\hyper@anchorstart{cite.#1}\hyper@anchorend}}

\def\my@lbibitem[#1]#2#3#4\par{%
	\item[\mybiblabel{#2}{#1}\myhyperanchor{#3}\hfill]#4%
	\@ifundefined{ifbackrefparscan}{}{\BR@backref{#3}}%
	\if@filesw{\let\protect\noexpand\immediate
		\write\@auxout{\string\bibcite{#3}{#1}}}\fi\ignorespaces%
}

\def\my@bibitem#1#2#3\par{%
	\refstepcounter\@listctr
	\item[\mybiblabel{#1}{\the\value\@listctr}\myhyperanchor{#2}\hfill]#3%
	\@ifundefined{ifbackrefparscan}{}{\BR@backref{#2}}%
	\if@filesw\immediate\write\@auxout
	{\string\bibcite{#2}{\the\value\@listctr}}\fi\ignorespaces%
}

\makeatother

\begin{document}
\title[Open Expanding Multi-valued Topological Dynamical Systems]{Ergodic Optimization for Open Expanding Multi-valued Topological Dynamical Systems}
\author{Oliver~Jenkinson, Xiaoran~Li, Yuexin~Liao, and Yiwei~Zhang}

\address[O.~Jenkinson]{School of Mathematical Sciences, Queen Mary University of London, Mile End Road,
London, E1 4NS, UK}
\email{o.jenkinson@qmul.ac.uk}

\address[X.~Li]{The Division of Physics, Mathematics and Astronomy, California Institute of Technology, 1200 E California Blvd, Pasadena CA 91125, USA}
\email{xiaoran@caltech.edu}

\address[Y.~Liao]{The Division of Physics, Mathematics and Astronomy, California Institute of Technology, 1200 E California Blvd, Pasadena CA 91125, USA}
\email{yliao@caltech.edu}

\address[Y.~Zhang]{School of Mathematical Sciences and big data, Anhui University of Science and Technology, Huainan, Anhui 232001, P.R. China}
\email{2024087@aust.edu.cn, yiweizhang831129@gmail.com}



\thanks{Y. Zhang is partially supported by National Natural Science Foundation (Nos.~12161141002, 12271432), and USTC-AUST Math Basic Discipline Research Center.}


\maketitle
\begin{abstract}
We study the optimization of ergodic averages for multi-valued dynamical systems, i.e.~where points may have multiple different forward orbits.
Under upper semi-continuity assumptions, we show that the maximum space average with respect to invariant probability measures for such systems can be characterised in terms of maximum time averages on an auxiliary shift space.
For all multi-valued expanding systems that are open mappings, we show that every H\"older continuous real-valued function can be modified by a coboundary, of the same H\"older exponent, such that the resulting function is dominated by its maximum ergodic average.
\end{abstract}

\section{Introduction}\label{introsection}

In this article we shall be concerned with dynamical systems where points have non-unique forward orbits.
Interest in such \emph{multi-valued} systems is motivated by the theory of overlapping iterated function systems
(see e.g.~\cite{bakermemams, baker, bakerkoivusalo, bms, hochman, peressimonsolomyak, pollicottsimon, rams, 
sidorov, solomyak}), by non-uniqueness and randomness in digit expansions (see e.g.~\cite{dk,dajanidevries, dajanidevries2,djkl}),
and by the dynamics of holomorphic correspondences
(see e.g.~\cite{bullett, bullettlomonaco, bullettpenrose,  LLMM21, LMM24}).

We shall be concerned with the ergodic theory of such systems: while there is an accepted
notion of invariant measure
(see e.g.~\cite{artstein,millerakin} and Section \ref{measures}) and a version of Poincar\'e's recurrence theorem
(see \cite{AFL91}), there is no general definition of ergodicity\footnote{For \emph{non-invariant} measures, a notion of ergodicity has been defined, see \cite{londhe}.}, and no direct analogue of
Birkhoff's ergodic theorem.
Despite the absence of these foundational concepts, 
our interest will be in ergodic averages, and more precisely
the prospect of giving meaning to the 
\emph{largest} (or \emph{smallest}) possible 
value of the ergodic average of a real-valued function defined on the state space of the dynamical system.
For single-valued dynamical systems,
such \emph{ergodic optimization} problems 
(see e.g.~\cite{Boc18, new survey} for overviews)
have been studied since the 1990s, initially stimulated
by numerical investigation of certain model problems (see \cite{Bou00, HO96a, HO96b, jenkinson1, jenkinson2}),
leading to the motivating paradigm of \emph{typical periodic optimization},
the conjecture (see \cite{HO96a,HO96b,YH99}), and eventual proof
(see \cite{Co16, GSZ25, HLMXZ, LZ25}) that for suitable chaotic systems, and suitably regular real-valued functions,
the corresponding ergodic averages are usually optimized by periodic orbits.

For single-valued systems,
 the most common theoretical setting for ergodic optimization
is a self-map $T$ of a compact metrizable space $X$, and a continuous function $f:X\to\R$;
 a $T$-invariant Borel probability measure $\mu$ is then said to be \emph{$f$-maximizing} if
$\int f\, \mathrm{d}\mu$ is equal to  
the maximum ergodic space average 
$$\beta(f)=\sup\left\{\int f\, d\nu: \nu\text{ is a $T$-invariant probability measure}\right\}.$$
The multi-valued dynamical systems
considered in this article
will consist of an upper
semi-continuous set-valued mapping $T$, from a compact metrizable space
$X$ to its power set $2^X$; the special case where $T(x)$ is a singleton for all $x\in X$ reduces to the familiar 
case of single-valued systems.

Perhaps surprisingly, in view of 
the absence of an analogue of Birkhoff's ergodic theorem in the multi-valued setting, we are able to define
a satisfactory notion of \emph{maximum ergodic average} for a continuous (or, more generally, upper semi-continuous)
real-valued function $f$.
At this level of generality we are able to prove (see Theorem \ref{thm:equivalentdefintions}) the coincidence of the maximum ergodic space average with various notions of maximum time average: this is done by appealing to an auxiliary \emph{orbit space}, equipped with the shift map, and working with time averages for this single-valued dynamical system.

Thereafter, we are motivated by the prospect of establishing a so-called \emph{Ma\~n\'e cohomology lemma}
for certain multi-valued systems. In the single-valued setting, this result 
(see e.g.~\cite{Bou00, CLT01, CG95,  LT03, Sav99})
asserts that, under suitable hypotheses, 
a coboundary $v-v\circ T$ can be added to the real-valued function $f$ so that the resulting function is 
weakly dominated by
its maximum ergodic average, leading to the characterisation of $f$-maximizing measures as precisely those supported by
the zero set of the function $f+v-v\circ T -\beta(f)$.
This technique may readily reveal the maximizing measure(s) for specific functions $f$
(see e.g.~\cite{adjr, aj, Bou00, jenkinsonhitting, announce, major, jenkinsonpams} for such usage), and the Ma\~n\'e lemma
is also a fundamental tool in the typical periodic optimization results of \cite{Co16, GSZ25, HLMXZ, LZ25}.

Remarkably, it turns out to be possible to define a suitable notion of coboundary, and establish
a version of the Ma\~n\'e cohomology lemma in the multi-valued setting, a result that
should underpin the investigation of typical periodic optimization (cf.~Conjecture \ref{TPOconj}) for classes
of multi-valued dynamical systems, and the identification of specific maximizing measures in concrete settings
(cf.~Conjecture \ref{sturmian_conj}).
Our route to proving this result (Theorems \ref{mane_open_expanding} and \ref{mane_X})
consists firstly of the observation that
the system is completely determined by its graph, and that a corresponding \emph{graph dynamical system} can be defined:
given $X$, $T$ and $f$, we then show that the ergodic optimization problem can be transported to the graph system,
in the sense that there exists an auxiliary real-valued function defined on the graph, with maximum ergodic average equal to that of $f$.
Our second step towards a Ma\~n\'e lemma 
is the development of a new theory of \emph{expanding} multi-valued dynamical systems,
generalising certain key features of the theory of single-valued expanding maps.
Combined with a hypothesis that these multi-valued systems be \emph{open} mappings (i.e.~open sets are mapped to open sets),
a satisfactory theory of inverse branches can be established, to an extent that facilitates the proof of
a Ma\~n\'e lemma.

Our Ma\~n\'e lemma is valid for open expanding multi-valued systems, and H\"older continuous real-valued functions.
Even in the single-valued setting (see Corollary \ref{mane_single_valued}), this result has something of a folklore status (see Section \ref{mane_section} for further details): to our knowledge it has not been stated in the generality given here,  
and certain of the proofs in the existing literature are incomplete, or apply to a more limited class of systems (e.g.~circle expanding maps, or subshifts of finite type).

The organisation of this article is as follows.
In Section \ref{measures} we review the theory of multi-valued dynamical systems and their invariant measures.
In Section \ref{maxergodic}, various notions of maximum ergodic average are defined, and shown to be equivalent
when the function $f$ is upper semi-continuous.
In Section \ref{graphsystem}, we investigate the relation between multi-valued dynamical systems and their graph systems,
and derive consequences for ergodic optimization.
In Section \ref{sec_openexpandingsetvalued}, a theory of open expanding multi-valued dynamical systems is developed,
and in Section \ref{mane_section} a Ma\~n\'e lemma is proved for such systems.
In Section \ref{examples_section} we consider 
ergodic optimization for certain specific functions, and 
specific open expanding multi-valued dynamical systems, with comparison to some single-valued analogues.

\section{Dynamical systems and their invariant measures}\label{measures}

This section consists of preliminaries on multi-valued dynamical systems, introducing the associated orbit space, giving various equivalent definitions of invariant measure, and establishing properties of the set of all invarant measures.

\begin{notation}
For a non-empty set $X$, let $2^X$ denote its power set, i.e.~the set of all subsets of $X$.
Given a map $T:X\to 2^X$, its \emph{graph} $gr(T)$ is defined by
\begin{equation}\label{graphdef}
gr(T) :=  \{(x,x')\in X^2: x'\in T(x)\}.
\end{equation}
Define projection maps $\pi, \pi':X^2\to X$ by
$$
\pi(x,x'):=x\quad ,\quad \pi'(x,x'):=x'\,,
$$
and let
\begin{equation}\label{projectionsgraph}
\pi_T := \pi|_{gr(T)}\quad , \quad \pi_T':= \pi'|_{gr(T)}
\end{equation}
be the restrictions of the two projection maps $\pi$, $\pi'$ to the graph of $T$.
\end{notation}

\begin{notation}
Given $T: X \rightarrow 2^X$, define $T^{-1} : X\rightarrow 2^X$ (cf.~e.g.~\cite[p.~34]{aubinfrankowska}) by 
$$T^{-1} (x') =\{ x \in X: x' \in T(x) \},$$
so that in particular $(T^{-1})^{-1} =T$.
\end{notation}

If $X$ is a compact metrizable space,
recall (see e.g.~\cite[p.~38]{aubinfrankowska}) that $T:X\to 2^X$
is \emph{upper semi-continuous} at a
point $x\in X$ if for every open set $V \subseteq X$ containing $T(x)$, there exists an open set $U \subseteq X$ containing $x$, such
that $T(u) \subseteq V$ for all $u\in U$; the map $T$ is \emph{upper semi-continuous} if it is upper semi-continuous at each
point of $X$.

\begin{rem}
The upper semi-continuity of $T:X\to 2^X$ is
equivalent to the upper semi-continuity of $T^{-1}:X\to 2^X$,
and these are equivalent to their graphs $gr(T)$, $gr(T^{-1})$ being closed subsets of $X^2$
(see e.g.~\cite[Thm.~16.12]{aliprantisborder}).

\noindent
Every map $T:X\to X$ induces a map $\tilde T:X\to 2^X$ by defining $\tilde T(x):=\{T(x)\}$, and $\tilde T$ is upper semi-continuous if and only if $T$ is continuous.
\end{rem}

We shall be interested in \emph{dynamical} properties of $T$, and adopt the following definition:

\begin{defn} (Multi-valued dynamical systems)

\noindent
For a compact metrizable space $X$,
and an upper semi-continuous map $T:X\to 2^X$,
the pair $(X,T)$ will be called a \emph{multi-valued topological dynamical system}.
The collection of all multi-valued topological dynamical systems will be denoted by $\D$.
\end{defn}

The dynamical character of $(X,T)$ stems from the following:

\begin{defn} (Orbit space)

\noindent
For $(X,T)\in\D$, 
define the \emph{orbit space}
\begin{equation}
    X_T := \{ (x_i)_{i \in \mathbb{Z}} \in X^{\mathbb{Z}} : x_{i+1} \in T (x_i) \text{ for all } i \in \mathbb{Z} \}.
\end{equation}
The spaces $X^{\mathbb{Z}}$ and $X_T$ are both compact and metrizable when equipped with the product topology. 
The \emph{shift map} $\sigma: X^{\mathbb{Z}} \rightarrow X^{\mathbb{Z}}$ defined by 
$$\sigma ((x_{i})_{i \in \mathbb{Z}}) :=(x_{i+1})_{i \in \mathbb{Z}}$$
is then continuous, with $\sigma (X_T) =X_T$.
\end{defn}

\begin{notation}\label{projection_notation}
For $(X,T)\in\D$, define $\pi_0 : X_T \rightarrow X$ by
\begin{equation*}
    \pi_0 ((x_i)_{i \in \mathbb{Z}}) :=x_0.
\end{equation*}
\end{notation}

\begin{notation}
For a compact metrizable space $X$, let $\B_X$ denote the collection of Borel subsets of $X$,
and let $\m(X)$ denote the set of Borel probability measures on $X$.
When equipped with the weak$^*$ topology, the space $\m(X)$ is compact.
We write $\supp \mu$ for the \emph{support} of a probability measure $\mu$, i.e.~the smallest closed set whose measure is equal to 1.
\end{notation}

\begin{defn}\label{def: invariant measure} (Invariant measures)

\noindent
For $(X,T)\in\D$, a measure $\mu\in \m(X)$ is said to be \emph{$T$-invariant} if
\begin{equation}\label{invariantmeasureineq}
\mu(A)\le \mu(T^{-1}A)
\end{equation}
 for all Borel subsets $A\in\B_X$.
\end{defn}

\begin{rem}\label{rem: invariant measure single-valued} 
Note that \eqref{invariantmeasureineq} gives $\mu (A) \leqslant \mu (T^{-1} A)$ and $\mu (X \setminus A) \leqslant \mu (T^{-1} (X \setminus A))$ for all $A \in \B_X$. 
  In the special case that $T$ is \emph{single-valued}, 
$T^{-1} (X \setminus A) =X \setminus T^{-1} A$, so $\mu (T^{-1} A) + \mu (T^{-1} (X \setminus A)) =1 =\mu (A) +\mu (X \setminus A)$,
so $\mu (A) =\mu (T^{-1} A)$ for all $A \in \B_X$, in other words $\mu$ is $T$-invariant in the classical sense
(see e.g.~\cite{walters}).
\end{rem}

There are a number of useful alternative characterisations of invariant measures\footnote{See
\cite{artstein}
for additional equivalent conditions.}:

\begin{lem}\label{invariantmeasure}
For $(X,T)\in\D$,
a measure $\mu\in \m(X)$ is $T$-invariant if and only if any of the following equivalent conditions hold:
\begin{itemize}
\item[(a)] $\mu(A)\le \mu(T^{-1}A)$ for all $A\in \B_X$.
\item[(b)] $\mu(A)\le \mu(TA)$ for all $A\in \B_X$.
\item[(c)] $\mu = \nu\circ \pi_0^{-1}$, where $\nu$ is a shift-invariant Borel probability measure on $X_T$.
\item[(d)] There exists a Borel probability measure $m$ on $gr(T)$
 such that
both marginal measures $m\circ\pi_T^{-1}$ and $m\circ(\pi_T')^{-1}$ are equal to $\mu$.
\item[(e)] There exists a measurable map
$k:X\to \m(X)$, $x\mapsto k_x$, with ${\rm supp}(k_x)\subseteq T(x)$ for $\mu$-almost every $x\in X$, such that $\mu(A)=\int k_x(A)\, d\mu(x)$ for all $A\in \B_X$.
\end{itemize}
\end{lem}
\begin{proof}
For the equivalence between (a), (c), (d), and (e), see \cite[Thm~3.2]{millerakin}. Note that (d) 
does not change if $T$ is replaced by $T^{-1}$,
whereas statement (a) becomes statement (b);
it follows that (b) is equivalent to (d), and thus all statements are equivalent.
\end{proof}

\begin{notation}
For $(X,T)\in\D$, let $\m_T=\m_T (X)$ denote the set of all $T$-invariant Borel probability measures, and $\m_\sigma =\m_\sigma (X_T)$ the set of all $\sigma$-invariant Borel probability measures on $X_T$.
\end{notation}

\begin{lem}\label{lem:MF}
If $(X,T)\in\D$ then $\m_T =\m_\sigma \circ \pi_0^{-1} :=\{ \nu \circ \pi_0^{-1} : \nu \in \m_\sigma \}$. 
The
set $\m_T$ is weak$^*$ compact, metrizable, and convex.
It is non-empty if and only if $X_T \neq\emptyset$.
\end{lem}
\begin{proof}
See \cite[Thm.~3.2]{millerakin} (cf.~\cite[Cor.~5.4]{millerakin})
for compactness, convexity, and the identity $\m_T =\m_\sigma \circ \pi_0^{-1}$,
and for the fact that $\m_T\neq\emptyset$ if and only if $X_T \neq\emptyset$.
To prove metrizability,
let $E$ be the vector space of continuous linear functionals on $C(X)$,
equipped with the weak$^*$ topology,
and let
$B$ be the closed unit ball  in $E$.
Since $X$ is compact and metrizable, $C(X)$ is separable,
so $B$
 is also weak$^*$ metrizable, by \cite[Thm.~10.7]{aliprantisburkinshaw}.
Since  $\m_T$ is a weak$^*$ closed subset of $B$, it is itself
metrizable.
\end{proof}

\begin{rem}\label{nonsimplex}
Suppose $(X,T)\in\D$. If $T$ is \emph{single-valued}
then the ergodic decomposition theorem asserts that $\m_T$ is a Choquet simplex (see e.g.~\cite{choquet, phelps}).
More generally this need not be the case: e.g.~if $X=\Z/4\Z$
and $T(x)=\{x+1,x-1\}$ for all $x\in X$, then the four measures of the form $\frac{1}{2}(\delta_x+\delta_{x+1})$
are the extreme points of $\m_T$, but
for example $\frac{1}{4}\sum_{x\in X}\delta_x$
can be expressed in more than one way as a convex average of the extreme measures.
\end{rem}

\begin{notation}
If $(X,T)\in\D$ and  $A\subseteq X$, let $T(A) :=\bigcup_{x \in A} T(x)$ and $T^{-1} (A) :=\bigcup_{x \in A} T^{-1} (x)$. 
The sets $T^n (A)$ and $T^{-n} (A)$ are then defined,
for all  $n \in \N$,
 by
\begin{equation}\label{inductive_definition}
    T^0 (A) :=A, \quad T^n (A) :=T(T^{n-1} (A)), \quad T^{-n} (A) :=T^{-1} (T^{-(n-1)} (A)) .
\end{equation}
\end{notation}

\begin{lem}\label{compact_image_iterates}
If $(X,T)\in\D$, and  $A\subseteq X$ is compact, then $T^n (A)$ is compact for all $n \in \mathbb{Z}$.
\end{lem}
\begin{proof}
If $A$ is a compact subset of $X$ then
$T(A)= \pi_T' (gr (T) \cap (A \times X))$ is compact,
so induction using (\ref{inductive_definition})
gives that $T^n (A)$ is compact for all $n \in \N$. A similar argument shows that $T^{-n} (A)$ is compact for all $n \in \N$,
so the result follows.
\end{proof}

The following is another useful criterion for $\m_T$ to be non-empty:

\begin{lem}\label{lem: dynamical domain}
    If $(X,T)\in\D$ then $\pi_0 (X_T) =\bigcap_{n \in \mathbb{Z}} T^n (X)$.
Moreover, 
$X_T \neq \emptyset$ if and only if $\bigcap_{n \in \mathbb{Z}} T^n (X) \neq \emptyset$.
\end{lem}
\begin{proof}
See \cite[Lemma~1.2]{millerakin}.
\end{proof}

\begin{cor}\label{wb0vvud0}
    Let $(X,T)\in\D$. If $T(X) =X$ or $T^{-1} (X) =X$, then $X_T \neq \emptyset$.
\end{cor}

\begin{proof}
    If $T^{-1} (X) =X$, then $\bigcap_{n \in \mathbb{Z}} T^n (X) =\bigcap_{n=0}^{\infty} T^n (X)$ is the intersection of a decreasing sequence of compact sets, hence non-empty,
so $X_T \neq \emptyset$
by Lemma~\ref{lem: dynamical domain}. The case $T(X) =X$ is proved similarly.
\end{proof}

\section{Maximum ergodic averages and maximizing measures}\label{maxergodic}

The purpose of this section is to develop suitable notions of maximum ergodic average, and maximizing measure,
in the context of multi-valued dynamical systems.
While the multiplicity of forward orbits for a given point is an obstacle to proving an analogue of Birkhoff's ergodic theorem,
in the context of the \emph{maximum} ergodic averages it turns out to be possible to equate space averages and time averages, provided the latter are interpreted on the orbit space equipped with the shift map.

\begin{notation}
Suppose $(X,T)\in\D$, and $n \in \N$. Define
\begin{equation}
    \cO_n (T) :=\left\{ (x_0,  x_1, \dots, x_n) \in X^{n+1} : x_{k+1} \in T(x_k) \text{ for } 0\le k\le n-1  \right\}.
\end{equation}
For $f : X \rightarrow \R$, we use $S_n f$ to denote both the function $S_n f : \cO_n (T) \rightarrow \R$
given by
	\begin{equation*}
		S_n f \left((x_i)_{i=0}^n\right) :=\sum_{k=0}^n f(x_k),
	\end{equation*}
and the function $S_n f : X_T \rightarrow \R$
given by
\begin{equation}\label{S_nf=}
		S_n f \left((x_i)_{i \in \Z}\right) :=\sum_{k=0}^n f(x_k),
	\end{equation}
as the argument of $S_nf$ will always be clear from the context.
Define $\tf \colon X_T \rightarrow \R$ by
	\begin{equation}\label{tf=}
		\tf \left((x_i)_{i \in \mathbb{Z}}\right) := f(x_0).
	\end{equation}
\end{notation}
	
\begin{defn}\label{maximum_ergodic_averages}
Given $(X,T)\in\D$, and $f:X\to\R$,
    an orbit $\udx \in X_T$ is said to be \emph{$(T, f)$-regular} if
$\lim_{n \to +\infty} \frac{1}{n+1} S_n f (\udx)$ exists as an element of $[-\infty, \infty]$.
The set of $(T, f)$-regular orbits will be denoted by $\Reg (X,T,f) =\Reg (T, f)$. 
With the convention that $\sup \emptyset =-\infty$,
define
	\begin{equation*}
		\beta (f) = \beta (f, T):= \sup_{\udx \in \Reg (T, f)} \lim_{n \to +\infty} \frac{1}{n+1} S_n f (\udx),
	\end{equation*}
	\begin{equation*}
		\gamma (f) = \gamma (f, T):= \sup_{\udx \in X_T} \limsup_{n \to +\infty} \frac{1}{n+1} S_n f (\udx),
	\end{equation*}
	\begin{equation*}
		\delta (f) = \delta (f, T) := \limsup_{n \to +\infty} \frac{1}{n+1} \sup_{(x_0 ,x_1 ,\dots ,x_n) \in \cO_n (T)} S_n f (x_0 ,x_1 ,\dots ,x_n),
	\end{equation*}
        \begin{equation*}
            \varepsilon (f) =\varepsilon (f,T) := \sup_{\udx \in X_T} \inf_{n \in \N} \frac{1}{n+1} S_n f(\udx),
        \end{equation*}
and if $f$ is Borel measurable and bounded either above or below, define
	\begin{equation}\label{alpha_defn}
		\alpha (f) = \alpha (f, T) := \sup_{\mu \in \m_T} \int \! f \, \mathrm{d} \mu.
	\end{equation}
\end{defn}
	
	In general the various notions of maximum ergodic average from Definition \ref{maximum_ergodic_averages}
need not be equal, but if $f$ is upper semi-continuous then it turns out that they are:
	
\begin{thm}\label{thm:equivalentdefintions}
    If $(X,T)\in\D$ and 
    $f \colon X \rightarrow \R$ is upper semi-continuous, then
    \begin{equation}\label{MEA for T}
		\alpha (f) =\beta (f) =\gamma (f) =\delta (f) =\varepsilon (f) \in [- \infty, +\infty ).
    \end{equation}
    If, moreover, $X_T \neq \emptyset$ and $f$ is bounded, then $\alpha (f) \in \R$.
\end{thm}
	
\begin{proof}
    If $X_T = \emptyset$ then $\Reg (T,f) =\emptyset$, thus $\beta (f) =\gamma (f) =\varepsilon(f) = -\infty$,
and $\m_T =\emptyset$
 by Lemma~\ref{lem:MF}, so $\alpha (f) =-\infty$. By Lemma~\ref{lem: dynamical domain}, $\bigcap_{n \in \mathbb{Z}} T^n (X) =\emptyset$, and since $T^n (X)$ is compact for all $n \in \mathbb{Z}$ by Lemma \ref{compact_image_iterates}, there exists $K \in \N$ with $\bigcap_{n=-K}^K T^n (X) =\emptyset$. It follows that $\cO_n (T) =\emptyset$ for all $n \geqslant 2K$ (since if on the contrary some $(x_0,x_1,\ldots,x_n)\in \cO_n(T)$ for $n \geqslant 2K$ then we would have $x_K \in \bigcap_{n=-K}^K T^n (X)$), therefore $\delta (f) =-\infty$, so \eqref{MEA for T} holds.

    Now suppose $X_T \neq \emptyset$, so that $\m_\sigma \neq \emptyset$ and $\m_T \neq \emptyset$ by Lemma~\ref{lem:MF}.
    Recalling that $\sigma \colon X_T \rightarrow X_T$ is the shift map, that $\m_\sigma$ is the set of all $\sigma$-invariant Borel probability measures on $X_T$, and that $\tf$ is as in (\ref{tf=}), we write
    \begin{equation*}
		\Reg \bigl(X_T, \sigma , \tf \bigr) := \biggl\{\udx \in X_T : \text{ the limit } \lim_{n\to +\infty} \frac{1}{n+1} \sum_{k=0}^{n} \tf (\sigma^{k} (\udx)) \text{ exists} \biggr\},
    \end{equation*}
    \begin{equation*}
		\beta \bigl( \tf , \sigma \bigr) := \sup_{\udx \in \Reg (X_T, \sigma ,  \tf)} \lim_{n \to +\infty} \frac{1}{n+1} \sum_{k=0}^{n} \tf (\sigma^k (\udx)),
    \end{equation*}
    \begin{equation*}
		\gamma \bigl( \tf , \sigma \bigr) := \sup_{\udx \in X_T} \limsup_{n \to +\infty} \frac{1}{n+1} \sum_{k=0}^{n} \tf (\sigma^k (\udx)),
    \end{equation*}
    \begin{equation*}
		\delta \bigl( \tf , \sigma \bigr) := \limsup_{n\to +\infty} \frac{1}{n+1} \sup_{\udx \in X_T} \sum_{k=0}^{n} \tf (\sigma^k (\udx)),
    \end{equation*}
    \begin{equation*}
        \varepsilon \bigl( \tf , \sigma \bigr) :=\sup_{\udx \in X_T} \inf_{n \in \N} \frac{1}{n+1} \sum_{k=0}^{n} \tf (\sigma^k (\udx)),
    \end{equation*}
    \begin{equation*}
		\alpha \bigl( \tf , \sigma \bigr) := \sup_{\nu \in \m_\sigma} \int_{X_T} \tf \, \mathrm{d} \nu.
    \end{equation*}
		
    Since $f$ is upper semi-continuous,  
so is $\tf$,
and \cite[Prop.~2.2]{jenkinsonergodicoptimization} gives
		\begin{equation*}\label{MEA for sigma}
			\alpha \bigl( \tf , \sigma \bigr) =\beta \bigl( \tf , \sigma \bigr) =\gamma \bigl( \tf , \sigma \bigr) =\delta \bigl( \tf , \sigma \bigr) \in [- \infty, +\infty ),
		\end{equation*}
    and \cite[Thm.~A.3]{5th} gives $\alpha \bigl( \tf , \sigma \bigr) =\varepsilon \bigl( \tf , \sigma \bigr)$
(cf.~also \cite[Sect.~3, eq.~(30)]{bochi}).
		
Now $\sum_{k=0}^{n} \tf (\sigma^k (\udx)) =S_n f (\udx)$
for all
$\udx \in X_T$ and $n\in \N$,  by (\ref{S_nf=}), (\ref{tf=}), 
therefore $\Reg \bigl( X_T, \sigma , \tf \bigr) =\Reg (X ,T ,f)$,  $\beta \bigl( \tf , \sigma \bigr) =\beta (f)$,  
$\gamma \bigl( \tf , \sigma \bigr) =\gamma (f)$,
and $\varepsilon \bigl(\tf,\sigma\bigr)=\varepsilon(f)$. 
        Note that 
$\int \! \tf \, \mathrm{d} \nu =\int \! f \, \mathrm{d} (\nu \circ \pi_0^{-1})$
for each $\nu \in \m_\sigma$, so $\alpha \bigl( \tf , \sigma \bigr) =\alpha (f)$ follows from 
the fact that $\m_T =\m_\sigma \circ \pi_0^{-1}$ (see Lemma~\ref{lem:MF}).
        For each $\udx =(x_i)_{i \in \mathbb{Z}} \in X_T$, note that $\sum_{k=0}^n \tf (\sigma^k (\udx)) =S_n f (x_0, \dots, x_n)$, so $\delta (f) \geqslant \delta \bigl( \tf , \sigma \bigr)$ by their definitions.
    It remains to prove that $\delta(f)\leqslant \alpha \bigl( \tf,\sigma \bigr)$, after which (\ref{MEA for T}) will follow.

        Now choose a strictly increasing sequence $\{ n_k \}_{k \in \N}$ of positive integers, and a corresponding sequence $\bigl( x_{\lfloor -n_k /2 \rfloor}^{(k)}, x_{\lfloor -n_k /2 \rfloor +1}^{(k)}, \dots, x_{\lfloor n_k /2 \rfloor}^{(k)} \bigr) \in \cO_{n_k} (T)$ such that
        \begin{equation}\label{sedelta}
            \lim_{k \to +\infty} \frac{1}{n_k +1} S_{n_k} f \bigl( x_{\lfloor -n_k /2 \rfloor}^{(k)}, x_{\lfloor -n_k /2 \rfloor +1}^{(k)}, \dots, x_{\lfloor n_k /2 \rfloor}^{(k)} \bigr) =\delta (f),
        \end{equation}
        where $\lfloor p \rfloor$ denotes the largest integer less than or equal to $p \in \R$.
        Choose $y_0 \in X$. We define $x_j^{(k)} =y_0$ when $j< \lfloor -n_k /2 \rfloor$ or $j >\lfloor n_k /2 \rfloor$ and set $\udx^{(k)} := \bigl( x_j^{(k)} \bigr)_{j \in \mathbb{Z}} \in X^{\mathbb{Z}}$. 
        Define
        \begin{equation}\label{nu^(k)}
            \nu^{(k)} := \frac{1}{n_k +1} \sum_{j=\lfloor -n_k /2 \rfloor}^{\lfloor n_k /2 \rfloor} \delta_{\sigma^j (\udx^{(k)})},
        \end{equation}
        so that \eqref{sedelta} gives
        \begin{equation*}
            \lim_{k \to \infty} \int_{X_T} \! \tf \, \mathrm{d} \nu^{(k)} =\delta (f).
        \end{equation*}
        The compactness of $X^{\mathbb{Z}}$ means the set of all Borel probability measures on $X^{\mathbb{Z}}$ is weak$^*$ compact, so $\{ \nu^{(k)} \}_{k \in \N}$ has an accumulation point $\nu$, which without loss of generality 
can be assumed to be the weak$^*$ limit of $\nu^{(k)}$ as $k \to \infty$. By 
\cite[Thm.~2.1]{billingsley}, the upper semi-continuity of $\tf$ implies
        \begin{equation}\label{u94g3efvidh4btrsefdvs}
            \int_{X_T} \! \tf \, \mathrm{d} \nu \geqslant \limsup_{k \to \infty} \int_{X_T} \! \tf \, \mathrm{d} \nu^{(k)} =\delta (f).
        \end{equation}
        From \eqref{nu^(k)} we see that $\nu$ is $\sigma$-invariant. To show that $\nu$ is supported on $X_T$, define the probability measure
        \begin{equation}\label{nu^k}
            \nu_k := \frac{1}{\lfloor (n_k -\sqrt{n_k})/2 \rfloor -\lfloor -(n_k -\sqrt{n_k})/2 \rfloor +1} \sum_{j=\lfloor -(n_k -\sqrt{n_k})/2 \rfloor}^{\lfloor (n_k -\sqrt{n_k})/2 \rfloor} \delta_{\sigma^j (\udx^{(k)})}
        \end{equation}
        for each $k \in \N$. Since $n_k \to \infty$ as $k \to \infty$, comparison of \eqref{nu^(k)} and \eqref{nu^k} gives
$\nu$ as the weak$^*$ limit of $\nu_k$ as $k \to \infty$, and
$\mu\mapsto \supp \mu$
 is a lower semi-continuous function (see \cite[Thm.~16.15]{aliprantisborder}, cf.~\cite[Prop.~8.1]{akin}), so
        \begin{equation*}
            \supp \nu \subseteq \liminf_{k \to +\infty} \supp \nu_k.
        \end{equation*}
        Thus, for each $\udy \in \supp \nu$, there is $N \in \N$ and a sequence $\{ \udy^k \}_{k \geqslant N}$, where $\udy^k \in \supp \nu_k$ for each $k \geqslant N$, which converges to $\udy$. By \eqref{nu^k}, $\udy^k =\sigma^j \bigl( \udx^{(k)} \bigr)$ for some $j \in \mathbb{Z}$ with $\lfloor -(n_k -\sqrt{n_k})/2\rfloor \leqslant j \leqslant \lfloor (n_k -\sqrt{n_k})/2 \rfloor$. 
Recalling that $\bigl( x_{\lfloor -n_k /2 \rfloor}^{(k)}, x_{\lfloor -n_k /2 \rfloor +1}^{(k)}, \dots, x_{\lfloor n_k /2 \rfloor}^{(k)} \bigr) \in \cO_{n_k} (T)$, if we write $\udy^k =(y^k_i)_{i \in \mathbb{Z}}$, then $(y^k_i, y^k_{i+1}) \in gr(T)$ for each $i \in \mathbb{Z}$ with $-\sqrt{n_k} /2 +1 \leqslant i \leqslant \sqrt{n_k} /2 -2$. 
Let us write $\udy =(y_i)_{i \in \mathbb{Z}}$. 
Since $\udy^k$ converges (in the product topology) to $\udy$ as $k \to \infty$  
then $y^k_i$ converges to $y_i$ as $k \to \infty$,  for each $i \in \mathbb{Z}$. 
Since $\lim_{k \to +\infty} n_k =\infty$, for each $i \in \mathbb{Z}$ we have that $(y^k_i, y^k_{i+1}) \in gr(T)$ for sufficiently large $k \in \N$. Then $(y_i, y_{i+1}) \in gr (T)$ follows from the fact that $gr (T)$ is closed in $X^2$, and hence $\udy \in X_T$. This proves that $\nu$ is supported on $X_T$,
so $\nu \in \m_\sigma$, and \eqref{u94g3efvidh4btrsefdvs} gives the required inequality $\delta (f) \leqslant \alpha \bigl( \tf , \sigma \bigr)$. Therefore (\ref{MEA for T}) holds.

        Finally, if $f$ is bounded then $\alpha (f) \in \R$ follows from the fact that $\m_T \neq \emptyset$.
	\end{proof}

\begin{defn}\label{alpha_max_erg_average} (Maximum ergodic average)

\noindent
If $(X,T)\in\D$ and $f : X \rightarrow \R$ is upper semi-continuous,  
we define the \emph{maximum ergodic average}, denoted by $\beta(X,T,f)$, or simply $\beta(f)$, 
to be any of the values in the equality (\ref{MEA for T}). In particular,
\begin{equation*}
		\beta (X,T,f) = \beta (f) = \sup_{\mu \in \m_T} \int \! f \, \mathrm{d} \mu.
	\end{equation*}
\end{defn}

\begin{rem}
By Remark~\ref{rem: invariant measure single-valued}, 
$\beta(f)$ as in Definition \ref{alpha_max_erg_average} coincides with the classical maximum ergodic average in the case that $T$ is single-valued.
\end{rem}

	\begin{defn}\label{def_maximizing} (Maximizing measures)

\noindent
		If $(X,T)\in\D$ with $X_T \neq \emptyset$, and $f \colon X \rightarrow \R$ is upper semi-continuous,
a measure $\mu \in \m_T$ is called \emph{$f$-maximizing} if
		\begin{equation*}
			\int \! f \, \mathrm{d} \mu =\beta (f).
		\end{equation*}
		Let $\m_{\max}(X,T,f)$, or simply $\m_{\max} (f)$, denote the set of $f$-maximizing measures.
	\end{defn}
	
	\begin{prop}\label{M_max(f)}
If $(X,T)\in\D$ with $X_T \neq \emptyset$, and $f \colon X \rightarrow \R$ is upper semi-continuous,
then $\m_{\max} (f)$ is a non-empty, compact, convex subset of $\m_T$, and each extreme point of $\m_{\max} (f)$ is an extreme point of $\m_T$.
	\end{prop}
	
	\begin{proof}
		By Lemma~\ref{lem:MF}, $\m_T \neq \emptyset$ follows from $X_T \neq \emptyset$. Moreover, by 
\cite[Thm.~2.1]{billingsley} and the upper semi-continuity of $f$, the real-valued functional $\mu \mapsto \int \! f \, \mathrm{d} \mu$ on $\m_T$ is upper semi-continuous. Consequently, $\m_{\max} (f)$, the preimage of its maximum, is non-empty and compact.

    By the definition of $f$-maximizing measures, if $\mu_1 ,\, \mu_2 \in \m_{\max} (f)$ and $a \in [0,1]$, then $a \mu_1 + (1-a) \mu_2 \in \m_{\max} (f)$. Hence $\m_{\max} (f)$ is convex.

		Let $\mu \in \m_{\max} (f)$ be an extreme point of
$\m_{\max} (f)$. 
If $\mu$ is not extreme in $\m_T$ then there exist $\mu_1, \mu_2 \in \m_T$, and $a \in (0,1)$, with $\mu_1 \neq \mu$, $\mu_2 \neq \mu$, and $\mu =a \mu_1 +(1-a) \mu_2$. Since $\int \! f \, \mathrm{d} \mu_1 \leqslant \beta (f)$, $\int \! f \, \mathrm{d} \mu_2 \leqslant \beta (f)$, and
$a \int \! f \, \mathrm{d} \mu_1 +(1-a) \int \! f \, \mathrm{d} \mu_2 =\int \! f \, \mathrm{d} \mu =\beta (f)$,
		we see that $\int \! f \, \mathrm{d} \mu_1 = \beta (f)$ and $\int \! f \, \mathrm{d} \mu_2 = \beta (f)$, in other words $\mu_1 ,\, \mu_2 \in \m_{\max} (f)$. This contradicts the fact that $\mu =a \mu_1 +(1-a) \mu_2$ is an extreme point of $\m_{\max} (f)$, therefore $\mu$ must in fact be an extreme point of $\m_T$.
	\end{proof}

\section{The graph dynamical system}\label{graphsystem}

In this section, we show that every dynamical system $(X,T)\in\D$ induces a new dynamical system on its graph $gr(T)$.
In particular, the maximum ergodic average of a given function on $X$ is shown
(see Proposition \ref{equal_max_ergodic_averages}) to be the same as that of a related
function on $gr(T)$,  a result with important consequences in the context of the Ma\~n\'e cohomology lemma
of Section \ref{mane_section}.

If $(X,T)\in\D$, we equip the graph $gr(T) =\{(x,x')\in X^2: x'\in T(x)\} \subseteq X^2$ with the subspace topology inherited from
the product topology on $X^2$,
so in particular $gr(T)$ is metrizable.
Recall that the projection maps $\pi_T, \pi_T' : gr(T) \rightarrow X$ are given by 
$$\pi_T (x,y) =x \ , \quad \pi_T'(x,y)=y .$$

\begin{defn} (Graph system)

\noindent
For $(X,T)\in\D$, define the \emph{graph system}
to be the pair $\left(gr(T), \hat{T}\right)$, where
$\hat{T}:gr(T)\to 2^{gr(T)}$ is defined by 
\begin{equation}\label{derivedequation}
\hat{T}:= \pi_T^{-1}\circ \pi_T'\,.
\end{equation}
\end{defn}

\begin{lem}
If $(X,T)\in\D$, then $\left(gr(T), \hat{T}\right)\in\D$.
\end{lem}
\begin{proof}
Since $\pi_T, \pi_T':gr(T)\to X$ are continuous, and also closed mappings (since $X$ is compact),
the map $\pi_T^{-1}:X\to 2^{gr(T)}$ is upper semi-continuous (see e.g.~\cite{jaynerogers}),
hence so is $\hat{T}$.
\end{proof}

In this section we will consider the relation between $(X,T)$ and $\left(gr(T), \hat{T}\right)$,
in particular the ergodic optimization problem on both of these systems.
To distinguish between the orbit spaces corresponding to $(X,T)$ and $\left(gr(T), \hat{T}\right)$,
we introduce the following notation:

\begin{notation}
 For $(X,T)\in \D$, with orbit space
\begin{equation*}
    X_T = \{ (\dots, x_{-1}, x_0, x_1, x_2, \dots) : x_{k+1} \in T( x_k) \text{ for all } k \in \mathbb{Z} \},
\end{equation*}
and corresponding orbit space for $(gr(T),\hat{T})$ being
\begin{equation*}
    gr(T)_{\hat{T}} := \{ (\dots, (x_{-1}, x_0), (x_0, x_1), (x_1, x_2), \dots) : x_{k+1} \in T(x_k) \text{ for all } k \in \mathbb{Z} \},
\end{equation*}
let $R_T : X_T \rightarrow gr(T)_{\hat{T}}$ be defined by
$$
R_T \left( (x_i)_{i\in\Z}\right) := \left((x_{i}, x_{i+1})\right)_{i\in\Z}.
$$
\end{notation}

It will be notationally convenient 
to distinguish between the shift maps acting on the two spaces
$(X,T)$ and $(gr(T), \hat{T})$ as follows:

\begin{notation}
If $(X,T)\in \D$, let
$\sigma_T$ 
denote the shift map $\sigma_T:X_T\to X_T$ defined by
$$
\sigma_T\left( (x_i)_{i\in\Z}\right) := (x_{i+1})_{i\in\Z},
$$
and let $\sigma_{\hat{T}}$ denote the shift map $\sigma_{\hat{T}}: gr(T)_{\hat{T}}\to gr(T)_{\hat{T}}$
defined by
$$
\sigma_{\hat{T}}\left( \left((x_{i-1}, x_i)\right)_{i\in\Z}\right)
:=
\left((x_{i}, x_{i+1})\right)_{i\in\Z}.
$$
\end{notation}

The shift maps acting on the corresponding orbit spaces are topologically conjugate:

\begin{lem}\label{top_conj}
For $(X,T)\in\D$, the map $R_T \colon X_T \rightarrow gr(T)_{\hat{T}}$ is a
topological conjugacy between $(X_T,\sigma_T)$ and $(gr(T),\sigma_{\hat{T}})$.
That is, $R_T$ is a
 homeomorphism
such that $$R_T \circ \sigma_T =\sigma_{\hat{T}} \circ R_T.$$
\end{lem}
\begin{proof}
This is a straightforward verification.
\end{proof}

Lemma \ref{top_conj} immediately gives a homeomorphism between the sets of shift-invariant probability measures
on the two orbit spaces:

\begin{cor}
If  $(X,T)\in\D$ then
\begin{equation}\label{0099f2u39u}
    \m_{\sigma_T} (X_T) =\m_{\sigma_{\hat{T}}} (gr(T)_{\hat{T}}) \circ R_T,
\end{equation}
where
$$
\m_{\sigma_{\hat{T}}} (gr(T)_{\hat{T}}) \circ R_T := 
\left\{\nu\circ R_T:\nu\in \m_{\sigma_{\hat{T}}} (gr(T)_{\hat{T}})\right\}.
$$
\end{cor}

It is convenient to clarify notation for the projection maps from the orbit spaces as follows:

\begin{notation}
For $(X,T)\in\D$,
recall that the projection $\pi_0 \colon X_T \rightarrow X$ is given by 
$$\pi_0 ((x_i)_{i \in \mathbb{Z}}) =x_0,$$ 
and let $\hat{\pi}_0$ denote the counterpart 
for the graph system $(gr(T), \hat{T})$, namely the
projection map  $\hat{\pi}_0 : gr(T)_{\hat{T}} \rightarrow gr(T)$ given by
\begin{equation*}
    \hat{\pi}_0 ((x_i, x_{i+1})_{i \in \mathbb{Z}}) :=(x_0, x_1).
\end{equation*}
\end{notation}

\begin{lem}
If $(X,T)\in\D$ then 
\begin{equation}\label{stafiv}
\m_T (X) =\m_{\sigma_T} (X_T) \circ \pi_0^{-1}
\end{equation}
 and 
\begin{equation}\label{stathr}
\m_{\hat{T}} (gr(T)) =\m_{\sigma_{\hat{T}}} (gr(T)_{\hat{T}}) \circ \hat{\pi}_0^{-1}.
\end{equation}
\end{lem}
\begin{proof}
This follows from Lemma~\ref{lem:MF}.
\end{proof}

\begin{notation}
Define $\pi_{01} \colon X_T \rightarrow gr(T)$ by
\begin{equation*}
    \pi_{01} ((x_i)_{i \in \mathbb{Z}}) = (x_0, x_1),
\end{equation*}
so that clearly
\begin{equation}\label{dubstar}
\pi_{01} =\hat{\pi}_0 \circ R_T.
\end{equation}
\end{notation}

\begin{lem}
 If $(X,T)\in\D$ then 
\begin{equation}\label{vjn9394031}
    \m_{\hat{T}} (gr(T)) 
=\m_{\sigma_T} (X_T) \circ \pi_{01}^{-1}.
\end{equation}
\end{lem}
\begin{proof}
By (\ref{stathr}) and  (\ref{dubstar}),
$$
\m_{\hat{T}} (gr(T)) =\m_{\sigma_{\hat{T}}} (gr(T)_{\hat{T}}) \circ \hat{\pi}_0^{-1}
= \m_{\sigma_{\hat{T}}} (gr(T)_{\hat{T}}) \circ R_T \circ \pi_{01}^{-1},
$$
and combining with \eqref{0099f2u39u} yields the required identity.
\end{proof}

\begin{lem}  If $(X,T)\in\D$ then 
\begin{equation}\label{mF=mhatFcircpi1^-1}
    \m_T (X) 
    = \m_{\hat{T}} (gr(T)) \circ \pi_T^{-1}.
\end{equation}
\end{lem}

\begin{proof}
Clearly
\begin{equation}\label{projections} 
\pi_0 =\pi_T \circ \pi_{01},
\end{equation}
so (\ref{stafiv}) and (\ref{projections}) give
$$
\m_T (X) =\m_{\sigma_T} (X_T) \circ \pi_0^{-1} = \m_{\sigma_T} (X_T) \circ \pi_{01}^{-1}\circ \pi_T^{-1},
$$
and combining with (\ref{vjn9394031}) gives the required identity.
\end{proof}

\begin{defn}\label{graph_function_defn} (Graph function)

\noindent
If $(X,T)\in\D$, and $f\in C(X)$, the \emph{graph function} 
 $\hat{f} : gr(T) \rightarrow \R$ is defined by 
\begin{equation}\label{fhatdef}
    \hat{f} (x_1, x_2) := f(x_1),
\end{equation}
so that $\hat{f}\in C(gr(T))$.
\end{defn}

Finally, we deduce the following relation between the ergodic optimization problems for $(X,T,f)$ and
$\left(gr(T),\hat{T},\hat{f}\right)$:

\begin{prop}\label{equal_max_ergodic_averages}
If $(X,T)\in\D$, and $f\in C(X)$, then
 $$\beta (X,T,f) =\beta \left(gr(T),\hat{T},\hat{f}\right),$$ 
$$\m_{\max} (X,T,f) =\m_{\max} \left(gr(T),\hat{T},\hat{f}\right) \circ \pi_T^{-1},$$
and
$$\m_{\max} \left(gr(T),\hat{T},\hat{f}\right) = \{ \nu \in \m_{\hat{T}} (gr(T)) :\nu \circ \pi_T^{-1} \in \m_{\max} (X,T,f) \}.$$
\end{prop}
\begin{proof}
Now $   \m_T (X) 
    = \m_{\hat{T}} (gr(T)) \circ \pi_T^{-1}$ by \eqref{mF=mhatFcircpi1^-1},
and since $\hat{f}=f\circ\pi$ by (\ref{fhatdef}), we have
\begin{equation}\label{integral_identity}
\int_{gr(T)} \! \hat{f} \, \mathrm{d} \nu =\int_X \! f \, \mathrm{d} (\nu \circ \pi_T^{-1}) 
\quad\text{for all }\nu \in \m_{\hat{T}}(gr(T)),
\end{equation}
so the result follows.
\end{proof}

\section{Open expanding systems}\label{sec_openexpandingsetvalued}

In this section we develop a theory of multi-valued dynamical systems that are both expanding (i.e.~distances
between sufficiently nearby points are expanded by some uniform factor) and open (i.e.~images of open sets are open),
motivated by the prospect of establishing, in Section \ref{mane_section}, a Ma\~n\'e cohomology lemma
for such systems.
In the single-valued setting, expanding maps have been studied since the 1960s, notably by Shub
\cite{shub1, shub2} in the differentiable setting
(see also e.g.~\cite{dekimpelee, epsteinshub, farrelljones, farrellgogolev,gogolevrodriguezhertz,gromov,hirsch,ks, leelee},
as well as \cite[Ch.~2]{baladi}, \cite[\S 1.7]{katokhasselblatt}, \cite[\S 4.11]{szlenk}).
In the context of general compact metric spaces, as considered here, 
the distance-expanding hypothesis by itself will not be sufficient to make progress, and an additional openness assumption
will be required\footnote{Consider, for example, the familiar setting of symbolic dynamics: in the case of subshifts $X$ equipped with a standard metric,
the shift map on $X$ is expanding, but is an \emph{open} mapping only when $X$ is of finite type (see \cite{parry}).}.
Open expanding maps
were introduced by Ruelle \cite[\S 7.26]{ruelle}, 
in the context of thermodynamic formalism,
and the first systematic and detailed exposition at this level of generality was by
Przytycki \& Urba\'nski \cite[Ch.~3]{PU} (see also 
Urba\'nski, Roy \& Munday \cite[Ch.~4]{URM} and Viana \& Oliveira \cite[Ch.~11]{vianaoliveira}).

\begin{defn} (Open multi-valued systems)

\noindent We say $(X,T)\in\D$ is \emph{open} if $T(A)$ is open in $X$ for every open subset $A \subseteq X$.
\end{defn}

A basic feature of open mappings is the following:

\begin{lem}\label{lowersemicontinuity of T^-1}
    If $(X,T)\in\D$ 
is open, then the cardinality map
$\# T^{-1} (\,\cdot\,) : X \rightarrow \mathbb{Z}_{\geq 0} \cup \{ \infty \}$ is lower semi-continuous.
\end{lem}

\begin{proof}
    Fixing an arbitrary sequence $y_n \in X$, that converges to $y \in X$, we must show that
    \begin{equation}\label{3948efhdyjbcjsad0}
        \# T^{-1} (y) \leqslant \liminf_{n \to +\infty} \# T^{-1} (y_n).
    \end{equation}
    Letting $l =\# T^{-1} (y)$, we can write $T^{-1} (y) = \{x_1,\, x_2,\, \dots ,\, x_l \}$, and choose $l$ pairwise disjoint open subsets $U_1 ,\, U_2 ,\, \dots, U_l$ of $X$ with $x_j \in U_j$ for each $1\le j\le l$. 
Since $(X,T)$ is open, $T(U_j)$ is an open neighbourhood of $y$  for each $1\le j\le l$, so $V :=\bigcap_{j=1}^l T(U_j)$ 
is also an open neighbourhood of $y$ in $X$. 
If $y' \in V$ then $T^{-1} (y') \cap U_j \neq \emptyset$  for each $1\le j\le l$ since 
$y' \in V \subseteq T(U_j)$,
therefore $\# T^{-1} (y') \geqslant l$
since $U_1 ,\, U_2 ,\, \dots U_l$ are pairwise disjoint. 
So there exists $N \in \mathbb{N}$ such that $\# T^{-1} (y_n) \geqslant l =\# T^{-1} (y)$ for all $n \geqslant N$,
and \eqref{3948efhdyjbcjsad0} follows.
\end{proof}

Henceforth we shall assume that $X$ is a metric (rather than metrizable) space,
and introduce the following notation.

\begin{notation}
Let $\D_{{\rm Met}}$ denote the collection of pairs $(X,T)$ where $X=(X,d)$ is a compact metric space,
and $T:X\to 2^X$ is upper semi-continuous.
\end{notation}

The following Definition \ref{def expanding} generalises the notion of
an  expanding (single-valued) map (cf.~\cite[Defn.~4.1.1]{URM}).
Our intuition for this definition is that a multi-valued expanding system should have the property that \emph{each forward branch is expanding}. The key consequence
(see Proposition~\ref{inverse branch of open expanding correspondences})
is that under the additional hypothesis of openness, all inverse branches turn out to be well-defined and contracting
(a fact that, in the single-valued setting, is known to have important applications for both
thermodynamic formalism \cite{PU,ruelle,URM}, facilitating transfer operator methods, and ergodic optimization \cite{Bou00,Co16,CLT01, new survey}).

\begin{defn}\label{def expanding} (Expanding multi-valued systems)

\noindent
The multi-valued dynamical system
    $(X,T)\in \D_{{\rm Met}}$ will be called \emph{expanding} (with respect to the metric $d$ on $X$) if there exist $\lambda >1$ and $\eta >0$ such that for all $x ,\, y \in X$, if $d(x,y) \leqslant \eta$, then
		\begin{equation*}
			\inf \{d(x' ,y') : x' \in T(x) ,\, y' \in T(y)\} \geqslant \lambda d(x,y).
		\end{equation*}
\end{defn}

\begin{rem}\label{separatedness of T^-1}
    Suppose $(X,T)\in \D_{{\rm Met}}$ is expanding, with parameters
$\lambda$, $\eta$ as in Definition~\ref{def expanding}. For each $x \in X$, the expanding property of $T$ means $d(y,z) >\eta$ for each pair of distinct points $y,\, z \in T^{-1} (x)$. Since $X$ is compact, the cardinality of any subset for which each pair of distinct members are distant by at least $\eta$ is bounded by a constant $M \in \mathbb{N}$, so $\# T^{-1} (x) \leqslant M$ for all $x \in X$, and in particular $\# T^{-1}(x)$ is always finite, so $\# T^{-1}(\,\cdot\,)$ is a 
mapping $X\rightarrow \mathbb{Z}_{\ge0}$.
\end{rem}

The following lemma,
a complement to Lemma \ref{lowersemicontinuity of T^-1},
represents a necessary condition for developing a satisfactory theory of inverse branches
(cf.~Proposition \ref{inverse branch of open expanding correspondences})
 for open expanding
multi-valued systems.

\begin{lem}\label{uppersemicontinuity of T^-1}
    If $(X,T)\in\D_{{\rm Met}}$ is expanding, then
the map $\# T^{-1} (\,\cdot\,) : X \rightarrow \mathbb{Z}_{\geq 0}$ is upper semi-continuous.
\end{lem}

    \begin{proof}
If $(X,T)$ is expanding with respect to $d$, then
    there exist $\lambda >1$, $\eta >0$
as in Definition~\ref{def expanding}. 
Fixing an arbitrary sequence $y_n\in X$ which converges to $y \in X$, we aim to show that
        \begin{equation}\label{3948efhdyjbcjsad}
            \# T^{-1} (y) \geqslant \# \limsup_{n \to +\infty} T^{-1} (y_n).
        \end{equation}
Let $l:=\# T^{-1} (y)$
and $T^{-1} (y) =\{ w_1 ,\, \dots ,\, w_l \}$,  and write $A= \{ (x,y_n) : n \in \mathbb{N}$ and $x \in T^{-1} (y_n)\}$,
noting that $A$ is a subset of the compact set $gr (T)$. Recall from Remark~\ref{separatedness of T^-1} that $T^{-1} (y_n)$ is a finite set for every $n \in \mathbb{N}$, so an accumulation point of $A$ must be of the form $(x,y) \in gr (T)$, i.e.~of the form $(w_i, y)$ for some $1\le i\le l$. 
Consequently, there exists $N \in \mathbb{N}$ such that $T^{-1} (y_n) \subseteq \bigcup_{i=1}^l U_X (w_i, \eta /2)$ 
for all $n \geqslant N$, where $U_X (w_i, \eta /2) := \{ w \in X : d(w,w_i) <\eta /2 \}$. Recall from Remark~\ref{separatedness of T^-1} that the distance between a pair of distinct points in $T^{-1} (y_n)$ is greater than $\eta$, so if $n \geqslant N$ and $1\le i\le l$, then the set $T^{-1} (y_n)$ 
contains at most one member of
$U_X (w_i, \eta /2)$.
        Hence $\# T^{-1} (y_n) \leqslant l =\# T^{-1} (y)$ for all $n \geqslant N$, and  (\ref{3948efhdyjbcjsad}) 
follows. 
    \end{proof}

Lemmas~\ref{lowersemicontinuity of T^-1} and~\ref{uppersemicontinuity of T^-1} immediately imply the following:

\begin{cor}\label{continuity of T^-1}
    If $(X,T)\in\D_{{\rm Met}}$ is open and expanding, then $\# T^{-1} (\,\cdot\,) : X \rightarrow \mathbb{Z}_{\geq 0}$ is continuous, 
i.e.~for each $x \in X$ there exists an open neighbourhood $U$ of $x$ with $\# T^{-1} (x') =\# T^{-1} (x) \in \mathbb{Z}_{\geq 0}$ for all $x' \in U$.
\end{cor}

The following key result describes the inverse branches of an open expanding 
system.

\begin{prop}\label{inverse branch of open expanding correspondences}
    Let $(X,T)\in\D_{{\rm Met}}$ be open, and expanding with respect to the metric $d$,
where $\lambda >1$ and $\eta >0$ satisfy that for all $x ,\, y \in X$, if $d(x,y) \leqslant \eta$ then
    \begin{equation*}
		\inf \{d(x' ,y') : x' \in T(x) ,\, y' \in T(y)\} \geqslant \lambda d(x,y).
    \end{equation*}
    Then there exists $M \in \mathbb{N}$ such that for every $x \in X$, there is a neighbourhood $U_x$ 
    of $x$ in $X$ 
    such that if $T^{-1} (x) =\emptyset$, then $T^{-1} (y) =\emptyset$ for all $y \in U_x$; 
and if $T^{-1} (x) \neq \emptyset$, then there exist continuous maps $S_i ,\, 1 \leqslant i \leqslant l_x \leqslant M$, satisfying
        \begin{enumerate}
            \smallskip
            \item The images $S_1 (U_x) ,\, S_2 (U_x) ,\, \dots ,\, S_{l_x} (U_x)$ are pairwise disjoint.
            \smallskip
            \item If $y \in U_x$ then $T^{-1} (y) =\{ S_1 (y) ,\, S_2 (y) ,\, \dots ,\, S_{l_x} (y) \}$.
            \smallskip
            \item If $1\le i\le l_x$ and $y,\, z \in U_x$, then
            \begin{equation}\label{contracting of inverse branch}
                d(S_i (y), S_i (z)) \leqslant \frac{1}{\lambda} d(y,z).
            \end{equation}
        \end{enumerate}
\end{prop}

\begin{proof}
    Recall from Remark~\ref{separatedness of T^-1} that there exists $M \in \N$ with $\# T^{-1} (x) \leqslant M$ for all $x \in X$.
    Fix an arbitrary $x \in X$. 
    If $T^{-1} (x) =\emptyset$, then by Lemma~\ref{uppersemicontinuity of T^-1}, we can choose a neighbourhood $U_x$ of $x$ satisfying $T^{-1} (y) =\emptyset$ for all $y \in U_x$.

    Now suppose that $T^{-1} (x) =\{ w_1 ,\, \dots ,\, w_{l_x} \} \neq \emptyset$, where $l_x =\# T^{-1} (x) \leqslant M$.
    For each $1\le i \le l_x$ we set $V_i := T\left(\overline{U_X (w_i, \eta /3)}\right)$, where 
$U_X (w_i, \eta /3) := \{ w \in X : d(w, w_i) <\eta /3 \}$.
    Then for each $y \in V_i$, there exists $z\in \overline{U_X (w_i, \eta /3)}$ with $y \in T(z)$. 

If $z ,\, w \in \overline{U_X (w_i, \eta /3)}$ satisfy $T(z) \cap T(w) \neq \emptyset$, then by the expanding property, 
either $z=w$ or $d(z,w) >\eta$; but $z ,\, w \in \overline{U_X (w_i, \eta /3)}$, so $d(z,w) \leqslant 2\eta /3$, 
and therefore $z=w$. Consequently, for each $y \in V_i$, there is exactly one $z \in \overline{U_X (w_i, \eta /3)}$ with $y \in T(z)$. This allows us to define $S_i \colon V_i \rightarrow \overline{U_X (w_i, \eta /3)}$ by $y \in T(S_i (y))$ for all $y \in V_i$. Then the reversal of its graph $\{ (S_i (y), y) : y \in V_i \}$ is $gr (T) \cap \left(\overline{U_X (w_i, \eta /3)} \times X\right)$, which is a closed subset of $\overline{U_X (w_i, \eta /3)} \times V_i$. This implies that $S_i$ is continuous, because $\overline{U_X (w_i, \eta /3)}$ is compact.

    Since $T$ is open, $V_i =T \left(\overline{U_X (w_i, \eta /3)}\right)$ is a neighbourhood of $x$ for each 
$1\le i\le l_x$, so $\bigcap_{i=1}^{l_x} V_i$ is a neighbourhood of $x$. Then by Corollary~\ref{continuity of T^-1}, we can choose an open neighbourhood $U_x$ of $x$ such that $U_x \subseteq \bigcap_{i=1}^{l_x} V_i$ and $\# T^{-1} (y) =l_x$ for all $y \in U_x$. 

Now let us verify (1), (2), and (3) for $U_x$ and the continuous maps $S_i$, $1\le i\le l_x$.

    Recall from Remark~\ref{separatedness of T^-1} that for each pair of distinct points $w_i ,\, w_j \in T^{-1} (x)$, we have $d(w_i, w_j) >\eta$, so the family $\left\{ \overline{U_X (w_i, \eta /3)} \right\}_{i=1}^{l_x}$ 
is pairwise disjoint. Now $S_i (U_x) \subseteq \overline{U_X (w_i, \eta /3)}$ for all 
 $1\le i\le l_x$, so $S_1 (U_x) ,\, S_2 (U_x) ,\, \dots ,\, S_{l_x} (U_x)$ are pairwise disjoint, and thus (1) holds.

    By the definitions of $S_i$, we have $\{ S_1 (y) ,\, S_2 (y) ,\, \dots ,\, S_{l_x} (y) \} \subseteq T^{-1} (y)$ for all $y \in U_x$. Recall that $\# T^{-1} (y) =l_x$ for all $y \in U_x$, 
so $T^{-1} (y)$ is precisely $\{ S_1 (y) ,\, S_2 (y) ,\, \dots ,\, S_{l_x} (y) \}$ for each $y \in U_x$,  and (2) follows.

    To prove (3), we fix arbitrary $y,\, z \in U_x$ and $1\le i\le l_x$. Since $S_i (y) ,\, S_i (z) \in \overline{U_X (w_i, \eta /3)}$, the distance between these points is less than $\eta$. Recall that $y \in T(S_i (y))$ and $z \in T(S_i (z))$,
so the expanding property of $T$ gives $d(y,z) \geqslant \lambda d(S_i (y), S_i (z))$, so (\ref{contracting of inverse branch}) holds.
\end{proof}

\begin{cor}\label{uniform locally inverse contracting branch}
    Under the same hypotheses as Proposition~\ref{inverse branch of open expanding correspondences}, there exists $\epsilon >0$ such that:

    If $y,\, z \in X$ with $d(y,z) <\epsilon$, then $\# T^{-1} (y) =\# T^{-1} (z)=l$, and if 
$l \neq 0$ then the sets can be written as $T^{-1} (y) =\{ x_1 ,\, \dots ,\, x_l \}$, $T^{-1} (z) =\{ w_1 ,\, \dots ,\, w_l \}$, with
    \begin{equation*}
        d(x_i, w_i) \leqslant \frac{1}{\lambda} d(y,z)\quad     \text{for each } 1\le i \le l.
    \end{equation*}
\end{cor}

\begin{proof}
    For each $x \in X$, choose a neighbourhood $U_x$ of $x$ as in Proposition~\ref{inverse branch of open expanding correspondences}. Since $X$ is compact, there is a Lebesgue number $\epsilon >0$ for the open cover $\{ U_x \}_{x \in X}$ of $X$. 
For each $y,\, z \in X$ with $d(y,z) <\epsilon$, we can find $x \in X$ with $y ,\, z \in U_x$. If $T^{-1} (x) =\emptyset$ then $T^{-1} (y) =T^{-1} (z) =\emptyset$, and thus $\# T^{-1} (y) =\# T^{-1} (z) =0$. Now suppose 
that $T^{-1} (x) \neq \emptyset$. By Proposition~\ref{inverse branch of open expanding correspondences} there are continuous maps $S_i$, $1\le i\le l$, satisfying (1), (2), (3) in Proposition~\ref{inverse branch of open expanding correspondences}, where $l= \# T^{-1} (x) >0$. Then $y,\, z \in U_x$ implies that $T^{-1} (y) =\{ S_1 (y) ,\, \dots ,\, S_l (y) \}$ and $T^{-1} (z) =\{ S_1 (z) ,\, \dots ,\, S_l (z) \}$. Property (3) means that for each $1\le i\le l$, we have $d(S_i (y), S_i (z)) \leqslant d(y,z)/\lambda$.
\end{proof}

\section{The  Ma\~n\'e cohomology lemma for open expanding multi-valued systems}\label{mane_section}

In single-valued ergodic optimization,
a \emph{Ma\~n\'e lemma}
refers to the existence of a coboundary that can be added to a function $f$, such that the resulting function is
no larger than $\beta(f)$;
the name was coined by Bousch 
\cite{Bou00}, as there had been analogous work of Ma\~n\'e
\cite{Ma92, Ma96}
in the context of Lagrangian flows.
A Ma\~n\'e lemma is known to be an important tool in proving \emph{typical periodic optimization} results (see \cite{Co16, HLMXZ, LZ25}, and cf.~Conjecture \ref{TPOconj}).
In this section,
working on the graph dynamical system, the coboundary
will be taken to mean a function of the form $v\circ\pi - v\circ\pi'$, where $\pi,\pi'$ are the coordinate projections.
We shall assume that
$(X,T)\in \D_{{\rm Met}}$ is open and expanding, and that $f$ is H\"older continuous.

\begin{notation}
For $(X,T)\in \D_{{\rm Met}}$, the graph $gr(T)$ will always be equipped with the metric $\hat{d}$ defined
by 
$$
\hat{d}\left( (x,y), (x',y') \right)
:=
\max \left( d(x,x'), d(y,y') \right).
$$
We shall consider H\"older continuous real-valued functions on both $X$ and $gr(T)$.
Recall that a real-valued function $g$ on a metric space $(Z,d_Z)$ is said to be 
\emph{$\alpha$-H\"older}, for $\alpha\in(0,1]$, if
there exists $K>0$ such that $|g(z)-g(z')|\le K d_Z(z,z')^\alpha$ for all $z,z'\in Z$.
\end{notation}

The following Ma\~n\'e lemma (Theorem \ref{mane_open_expanding}) is for open expanding  $(X,T)\in\D_{{\rm Met}}$, where the H\"older function
$f$ is defined on the graph $gr(T)$.
As noted in Section \ref{introsection}, we anticipate that this result will be a foundational step for proving results 
identifying specific maximizing measures (cf.~Conjecture \ref{sturmian_conj}), and in establishing that
maximizing measures are \emph{typically periodic} (cf.~Conjecture \ref{TPOconj}).
Moreover, the context of arbitrary open expanding systems is, even in the single-valued setting,
somewhat more general than is usual in the literature. 
A consequence of this generality is that the proof is rather longer\footnote{The 
length of proof rather belies the title of Ma\~n\'e \emph{lemma}, a situation reminiscent of that for
the celebrated \emph{Closing Lemma} \cite{pugh}, which was initially assumed to be a a straightforward
result.} 
than some of those that are available in special cases.

Some of the first proofs of Ma\~n\'e lemmas were given for expanding maps of the circle (see \cite{Bou00,CLT01}), 
a setting that allows certain simplifications.
On the one hand such maps are always topologically transitive (as are all smooth expanding maps of compact connected manifolds, see e.g.~\cite[Thm.~6.4.2]{URM}); transitivity\footnote{Other versions of the Ma\~n\'e lemma, such as in \cite{jenkinsonmauldinurbanskisft, Sav99} make hypotheses that imply transitivity.}
allows a proof using a nonlinear analogue of a Ruelle transfer operator  introduced in \cite{Bou00}
(the operator involves maximizing, rather than averaging, over preimages),
 whereas for non-transitive systems this operator need not have a fixed point\footnote{So for example
 the Ma\~n\'e lemma as stated in \cite[Prop~2.2]{Co16} is not correct in full generality, nor is
the preceding \cite[Lem~2.1(1)]{Co16}.}.
On the other hand, connectedness of the circle leads to other simplifications.
Our proof of Theorem \ref{mane_open_expanding} is partly based on the sketch proof in \cite{new survey},
and a significant portion (specifically, the proof of Claim 3) is dedicated to showing that a 
certain function $\phi$ does not take the value $+\infty$ (a point elided in \cite{new survey}); by contrast if 
$X$ is connected
it suffices to show that the continuous function $\phi$ is finite at a single point, leading to a considerably shorter proof. 

	\begin{thm}\label{mane_open_expanding}
Suppose $(X,T)\in \D_{{\rm Met}}$ is open and expanding. Let $\alpha \in (0,1]$. For every $\alpha$-H\"{o}lder function $f : gr(T) \rightarrow \R$, there exists an $\alpha$-H\"{o}lder function $v : X \rightarrow \R$ such that 
\begin{equation}\label{mane_required_inequality}
f +v \circ \pi -v\circ \pi' \le \beta\left(gr(T), \hat{T},f \right).
\end{equation}
	\end{thm}

    \begin{proof} We will write $\beta(f)$ instead of $ \beta(gr(T), \hat{T},f)$.
        For each $y \in X$ and $n \in \N$, define
        \begin{equation*}
            \cO_{-n} (y) := \{ (x_{-n}, \dots , x_{-1}, x_0) \in X^{n+1} : x_0 =y \text{ and } x_{k+1} \in T(x_k) \text{ for all } -n \le k \le -1  \}.
        \end{equation*}
    
        Define the function $\phi : X \rightarrow \R \cup \{ -\infty ,\, +\infty\}$ by
        \begin{equation}\label{phi=}
            \phi (x) := \sup \biggl\{ \sum_{k=-n}^{-1} f(x_k, x_{k+1}) -n \beta (f) : n \in \N ,\, (x_{-n}, \dots ,x_{-1}, x_0) \in \cO_{-n} (x) \biggr\},
        \end{equation}
        where we follow the convention that $\sup \emptyset =-\infty$. Clearly, $\phi (x) =-\infty$ if and only if $T^{-1} (x) =\emptyset$, i.e.~$\# T^{-1} (x) =0$. By Corollary~\ref{continuity of T^-1}, the set $\phi^{-1} (-\infty)$
is a strictly positive distance from its complement $X \setminus \phi^{-1} (-\infty)$,
in other words
$$\delta:=\inf\left\{d(x,y):x\in \phi^{-1} (-\infty), y\notin \phi^{-1} (-\infty)\right\}>0.$$

        The proof will begin by successively establishing four Claims concerning the function $\phi$.

        \smallskip
        \emph{Claim 1.} There exist $C, \epsilon >0$ such that if $x ,\, y \in X$ with $d(x,y) <\epsilon$ then exactly one of the following holds:
        \begin{enumerate}
            \smallskip
            \item $\phi (x) =\phi (y) =-\infty$.
            \smallskip
            \item $\phi (x) =\phi (y) =+\infty$.
            \smallskip
            \item $\phi (x) ,\, \phi (y) \in \R$ and $\abs{\phi (x) -\phi (y)} \leqslant C d(x,y)^\alpha$.
        \end{enumerate}

      For each $x\in X$, let $U_x$ be a neighbourhood of $x$ as in Proposition~\ref{inverse branch of open expanding correspondences}; we furthermore assume each $U_x$ to be open.
Let $\epsilon\in(0,\delta)$  be a Lebesgue number for the open cover $\{ U_x \}_{x \in X}$. 
It is then sufficient to find a constant $C >0$ and prove that if $\phi (x) \in \R$ and $d(x,y) <\epsilon$, then $\phi (y) \in \R$ and $\abs{\phi (x) -\phi (y)} \leqslant C d(x,y)^\alpha$.

        Fix $x ,\, y \in X$ with $\phi (x) \in \R$ and $d(x,y) <\epsilon$. We aim to show 
that $\abs{\phi (x) -\phi (y)} \leqslant C d(x,y)^\alpha$ for some $C >0$. We have 
that $\phi (y) >-\infty$, since $\phi (x) \in \R$ and $d(x,y) <\epsilon$, therefore $T^{-1} (y) \neq \emptyset$.

        We first fix arbitrary $n \in \N$ and $(y_{-n}, \dots ,y_{-1}, y_0) \in \cO_{-n} (y)$. Since $d(x,y)$ is smaller than a Lebesgue number for the open cover $\{ U_x \}_{x \in X}$, Proposition~\ref{inverse branch of open expanding correspondences} implies that there is a continuous map $S$ defined on a neighbourhood of both $x$ and $y$, such that $S(y)= y_{-1}$, $S(x) \in T^{-1} (x)$, and $d(S(x), S(y)) \leqslant d(x,y) /\lambda$ for some constant $\lambda >1$. If we set $x_0 =x$, then the argument above allows us to choose $x_{-1} \in T^{-1} (x_0)$ such that $d(x_{-1}, y_{-1}) \leqslant d(x_0, y_0)/\lambda$, and the same argument shows that we can inductively choose $x_{-(k+1)} \in T^{-1} (x_{-k})$ with 
$$d(x_{-(k+1)}, y_{-(k+1)}) \leqslant d(x_{-k}, y_{-k}) /\lambda\text{ for all }0\le k\le n-1.$$ 
In this way, we get an orbit $(x_{-n}, \dots ,x_{-1}, x_0) \in \cO_{-n} (x)$ such that $d(x_k, y_k) \leqslant \lambda^k d(x,y)$ for all $-n \le k\le 0$. 
Since $f$ is $\alpha$-H\"{o}lder, there exists $K>0$ such that
        \begin{equation*}
            \abs{f(z_1, w_1) -f(z_2, w_2)} \leqslant K d_2 ((z_1, w_1), (z_2, w_2))^\alpha =K \max \{ d(z_1, z_2),\, d(w_1, w_2) \}^\alpha
        \end{equation*}
        for all $(z_1, w_1) ,\, (z_2, w_2) \in gr(T)$, so
        \begin{alignat*}{4}
            &\biggl( \sum_{k=-n}^{-1} f(y_k, y_{k+1}) -n\beta (f) \biggr) -\biggl( \sum_{k=-n}^{-1} f(x_k, x_{k+1}) -n\beta (f) \biggr)\\
            &\qquad =\sum_{k=-n}^{-1} (f(y_k, y_{k+1}) -f(x_k, x_{k+1}))\\
            &\qquad \leqslant \sum_{k=-n}^{-1} K \max \{ d(x_k, y_k) ,\, d(x_{k+1}, y_{k+1}) \}^\alpha\\
            &\qquad \leqslant \sum_{k=-n}^{-1} K (\lambda^{k+1} d(x,y))^\alpha\\
            &\qquad \leqslant \frac{K}{1-\lambda^{-\alpha}} d(x,y)^\alpha.
        \end{alignat*}
        Set $C := K/(1- \lambda^{-\alpha})$. Since $n \in \N$ and $(y_{-n}, \dots ,y_0) \in \cO_{-n} (y)$ were chosen arbitrarily, the above inequality gives
        \begin{equation}\label{poikjhds}
            \phi (x) \geqslant \phi (y) -C d(x,y)^\alpha.
        \end{equation}

        Since $\phi (x) \in \R$ and $\phi (y) >-\infty$, (\ref{poikjhds}) implies that $\phi (y) \in \R$.
An argument analogous to the proof of (\ref{poikjhds}) shows that $\phi (y) \geqslant \phi (x) -C d(x,y)^\alpha$, 
and consequently $\abs{\phi (x) -\phi (y)} \leqslant Cd(x,y)^{\alpha}$, so Claim 1 is proved.

        \smallskip
        \emph{Claim 2.} There exists $M>0$ such that $\phi (X) \subseteq [-M, M] \cup \{ -\infty ,\, +\infty \}$.

        If $\R \cup \{ -\infty ,\, +\infty \}$ is topologised such that $\R$ has the Euclidean topology, and $-\infty$ and $+\infty$ are isolated points, then $\phi : X \rightarrow \R \cup \{ -\infty ,\, +\infty \}$ is continuous by Claim 1,
so $\phi (X)$ is compact in $\R \cup \{ -\infty ,\, +\infty \}$, and Claim 2 follows.

        \smallskip
        \emph{Claim 3.}  $\phi^{-1} (+\infty) =\emptyset$.

        Let us assume, for a contradiction, that $A:= \phi^{-1} (+\infty)$ is non-empty.
        By Claim 2, there exists $M>0$ such that $\phi (x) \leqslant M$ for all $x \in X \setminus A$. Fix an arbitrary $y \in A$. Since $\phi (y) =+\infty$, 
        we can choose $m_y \in \N$ and $(x_{-m_y}, \dots ,x_{-1}, x_0) \in \cO_{-m_y} (y)$ such that
        \begin{equation}\label{qcds}
            \sum_{k=-m_y}^{-1} f(x_k, x_{k+1}) -m_y \beta (f) >1+M +\max f -\beta (f).
        \end{equation}
        By Proposition~\ref{inverse branch of open expanding correspondences}, for each $-m_y +1 \le k \le 0$, there is an open neighbourhood $U_k$ of $x_k$, and a continuous map $S_k : U_k \rightarrow X$ satisfying $S_k (x_k) =x_{k-1}$ and $S_k (z) \in T^{-1} (z)$ for all $z \in U_k$. For each $1\le j \le m_y$, set $S^j_y := S_{-j+1} \circ \cdots \circ S_{-1} \circ S_0$, and write $S^0_y := \operatorname{Id}$. Then there is a neighbourhood $U$ of $y$ such that
if $0\le  j \le m_y$
then $S^j_y$
 can be defined on $U$, and is continuous, with $S^j_y (y) =x_{-j}$. 
Consequently, the map $\Psi_y$ on $U$ given by
        \begin{equation*}
            \Psi_y (z) := \sum_{j=0}^{m_y-1} f\bigl(S^j_y (z), S^{j+1}_y (z) \bigr)
        \end{equation*}
        is continuous, and by (\ref{qcds}), $\Psi_y (y) >m_y \beta (f) +K$, where $K:= 1+M +\max f -\beta (f)$. This allows us to choose an open neighbourhood $W_y$ of $y$ such that $\Psi_y (z) >m_y \beta (f) +K$ for all $z \in W_y$.

        For $z \in W_y$, write $z_{-j} := S^j_y (z)$ for all $0\le j \le m_y$,
so that $(z_{-m_y}, \dots ,z_{-1}, z_0) \in \cO_{-m_y} (z)$ and
        \begin{equation}\label{-m to 0}
            \sum_{k=-m_y}^{-1} f(z_k, z_{k+1}) -m_y \beta (f) >K.
        \end{equation}
        Now if $z_{-m_y} \notin A$ then we choose $0\le n \le m_y -1$ such that $z_{-n} \in A$ and $z_{-n-1} \notin A$. Then we have $\phi (z_{-n-1}) \leqslant M$, so that
        \begin{equation}\label{-m to -n-1}
            \sum_{k=-m_y}^{-n-2} f(z_k, z_{k+1}) -(m_y -n-1) \beta (f) \leqslant M.
        \end{equation}
        Moreover, we have
        \begin{equation}\label{-n-1 to -n}
            f(z_{-n-1}, z_{-n}) -\beta (f) \leqslant \max f -\beta (f),
        \end{equation}
        so the inequalities (\ref{-m to 0}), (\ref{-m to -n-1}), and (\ref{-n-1 to -n}) yield
        \begin{equation}\label{ws9cid}
            \sum_{k=-n}^{-1} f(z_k, z_{k+1}) -n \beta (f) >1.
        \end{equation}
        If on the other hand $z_{-m_y} \in A$ then we set $n =m_y$, so that (\ref{ws9cid}) holds by (\ref{-m to 0}). 

So we have shown that we can always find $n \leqslant m_y$ and $(z_{-n}, \dots ,z_{-1}, z_0) \in \cO_{-n} (z)$ satisfying $z_{-n} \in A$ and (\ref{ws9cid}).

        By Claim 1, $A= \phi^{-1} (+\infty)$ is closed in $X$, so it is compact. Since $\{ W_y \}_{y \in A}$ is an open cover of $A$, we can choose $L \in \N$ and $y_1 ,\, \dots ,\, y_L \in A$ such that $A \subseteq \bigcup_{i=1}^L W_{y_i}$. For convenience, we write $W_i =W_{y_i}$ and $m_i =m_{y_1}$ for all $1\le i\le L$. 
Set $N := \max \{ m_i : 1\le i\le  L \}$. Then for an arbitrary $z \in A$, there exists $1\le i \le L$ with $z \in W_i$, so we can find $n \leqslant m_i \leqslant N$ and $(z_{-n}, \dots ,z_{-1}, z_0) \in \cO_{-n} (z)$ satisfying $z_{-n} \in A$ and (\ref{ws9cid}).

        Let $x^1_0 \in A$. By the above, we can inductively find a sequence of positive integers $\{ n_k \}_{k \in \N}$, 
and a sequence of orbits $\{ (x^k_{-n_k}, \dots, x^k_{-1}, x^k_0) \}_{k \in \N}$, 
such that
for all $k \in \N$ we have
 $n_k \leqslant N$ and $x^k_{-n_k} \in A$, and  
$$(x^{k+1}_{-n_{k+1}}, \dots, x^{k+1}_{-1}, x^{k+1}_0) \in \cO_{-n_{k+1}} (x^k_{-n_k}),$$ 
and
        \begin{equation}\label{sjkdnoipd}
            \sum_{j=-n_k}^{-1} f(x^k_j, x^k_{j+1}) -n_k \beta (f) >1.
        \end{equation}
         For each $k \in \N$, define a probability measure $\mu_k$ on $gr(T)$ by
        \begin{equation*}
            \mu_k := \frac{1}{n_1 +n_2 + \cdots +n_k} \sum_{i=1}^k \sum_{j=-n_k}^{-1} \delta_{(x^i_j, x^i_{j+1})},
        \end{equation*}
        where $\delta_{(x^i_j, x^i_{j+1})}$ denotes the Dirac measure at $\bigl( x^i_j, x^i_{j+1} \bigr)$. Then by (\ref{sjkdnoipd}), we have
        \begin{equation*}
            \int_{gr(T)} \! f \, \mathrm{d} \mu_k >\frac{1}{n_1 +\cdots +n_k} \sum_{i=1}^k (n_i \beta (f) +1) =\beta (f) +\frac{k}{n_1 +\cdots +n_k} \geqslant \beta (f) +\frac{1}{N},
        \end{equation*}
        and since the space of probability measures on $gr(T)$ is weak* compact, the sequence $\{ \mu_k \}_{k \in \N}$ has an accumulation point $\mu$. Then $\mu \in \m_{T}$ and $\int \! f \, \mathrm{d} \mu \geqslant \beta (f) +1/N$. This contradicts the definition of $\beta (f)$. Therefore, $A =\emptyset$.

        \smallskip
        \emph{Claim 4.} $f+ \phi \circ \pi -\phi \circ \pi' \leqslant \beta (f)$.

        Fix an arbitrary $(x,y) \in gr(T)$. By (\ref{phi=}), for each $n \in \N$ and $(x_{-n}, \dots ,x_0) \in \cO_{-n} (x)$, we have $(x_{-n}, \dots ,x_0, y) \in \cO_{-(n+1)} (y)$, so
        \begin{equation*}
            \sum_{k=-n}^{-1} f(x_k, x_{k+1}) -n\beta (f)\leqslant \phi (y) +\beta (f) -f(x,y).
        \end{equation*}
        Since $n \in \N$ and $(x_{-n}, \dots ,x_0) \in \cO_{-n} (x)$ are chosen arbitrarily, we obtain that 
$$\phi (x) \leqslant \phi (y) +\beta (f) -f(x,y).$$ 
Recalling that $\phi (z) =-\infty$ if and only if $T^{-1} (z) =\emptyset$, note that $T^{-1} (y)$ is non-empty since 
$x \in T^{-1} (y)$, therefore $\phi (y) >-\infty$. Claim 3 gives $\phi (y) \in \R$, so 
$f(x,y) -\phi (y) +\phi (x) \leqslant \beta (f)$, and since $\pi (x,y) =x$ and $\pi' (x,y) =y$ then Claim 4 follows.

Having proved Claims 1--4, we now use them to conclude the proof.

        By Claim 3, if $x \in X$ then either $\phi (x) =-\infty$ or $\phi (x) \in \R$. By Claim 2, there exists $M>0$ such that $\phi (x) \in \{-\infty\} \cup [-M,M]$ for all $x \in X$.
        Now define $v : X \rightarrow \R$ by
        \begin{equation}\label{v=}
            v(x) := \begin{cases}
                \phi (x), & \phi (x) \in \R,\\
                -M -\max f +\beta (f), & \phi (x) =-\infty.
            \end{cases}
        \end{equation}
        Recalling that 
$\phi^{-1} (-\infty)$
is a strictly positive distance from its complement $X \setminus \phi^{-1} (-\infty)$,
we see that $v$ is $\alpha$-H\"{o}lder by Claim 1. Finally, it remains to check that
        \begin{equation}\label{lspakpapps}
            f(x,y) +v(x) -v(y) \leqslant \beta (f)
        \end{equation}
        for all $(x,y) \in gr(T)$.
For this, fix an arbitrary $(x,y) \in gr(T)$. Since $T^{-1} (y) \neq \emptyset$, by (\ref{v=}) we have $v(y)= \phi (y) \in \R$. If $\phi (x) \in \R$ then $v(x) =\phi (x)$, and thus (\ref{lspakpapps}) follows from Claim 4. If on the other hand
$\phi (x) =-\infty$, then (\ref{v=}) implies that
        \begin{equation*}
            v(x) =-M -\max f +\beta (f) \leqslant v(y) -f(x,y) +\beta (f),
        \end{equation*}
       which is the required inequality (\ref{lspakpapps}).
So we have checked that (\ref{lspakpapps}) holds for all $(x,y) \in gr(T)$, 
which is the required inequality (\ref{mane_required_inequality}).
    \end{proof}

We now deduce a Ma\~n\'e lemma for open expanding $(X,T)\in\D_{{\rm Met}}$, where the function
$f$ is defined on $X$.

\begin{thm}\label{mane_X}
Suppose $(X,T)\in\D_{{\rm Met}}$ is open and expanding.
For any $\alpha$-H\"older function $f:X\to\R$, there exists an $\alpha$-H\"older $v:X\to\R$ such that
\begin{equation}\label{fvineq}
f(x)+v(x) - v(y) \le \beta(X,T,f) \ \text{ for all }(x,y)\in gr(T).
\end{equation}
\end{thm}
\begin{proof}
Given $f:X\to\R$, define
$\hat{f} : gr(T) \rightarrow \R$ by 
$
    \hat{f} (x_1, x_2) = f(x_1)
$,
as in Definition \ref{graph_function_defn},
and note that $\hat{f}$ is $\alpha$-H\"older (with respect to $\hat{d}$), since $f$ is $\alpha$-H\"older with respect to $d$.
so Theorem \ref{mane_open_expanding} implies that there exists an $\alpha$-H\"older function
$v : X \rightarrow \R$ such that 
\begin{equation}\label{fhat_v_ineq}
\hat{f} +v \circ \pi -v\circ \pi' \le \beta\left(gr(T), \hat{T}, \hat{f} \right).
\end{equation}
However, $\beta \left(gr(T),\hat{T},\hat{f}\right) = \beta (X,T,f)$
by Proposition \ref{equal_max_ergodic_averages}, so (\ref{fhat_v_ineq}),
together with the definitions of $\hat{f}$, $\pi$ and $\pi'$, gives the required inequality (\ref{fvineq}).
\end{proof}

In particular, we deduce the following
special case of Theorem \ref{mane_X} for single-valued 
systems, in the generality of arbitrary
open expanding maps:

\begin{cor}\label{mane_single_valued}
Let $X$ be a compact metric space, and $T:X\to X$ an open expanding map.
For any $\alpha$-H\"older function $f:X\to\R$, there exists an $\alpha$-H\"older $v:X\to\R$ such that
$f(x)+v(x)-v(T(x))\le \beta(X,T,f)$ for all $x\in X$.
\end{cor}

\section{Some examples of ergodic optimization for open expanding multi-valued dynamical systems}\label{examples_section}

We conclude by briefly considering certain (classes of) multi-valued dynamical systems 
covered by the theory developed in previous sections.

In general, suppose $(X,d)$ is a compact metric space, $C_1,\ldots, C_n$ are closed subsets of $X$,
and for each $1\le i\le n$ there is a continuous open mapping $T_i:C_i\to X$
that is expanding with respect to $d$ in the sense that there exists $\lambda_i>1$ such that
$d(T_i(x),T_i(y))\ge \lambda_i d(x,y)$ for all $x,y\in X$.
The multi-valued map $T$ given by $T(x)=\{T_i(x):x\in C_i\}$ is open, and is expanding in the sense of Definition
\ref{def expanding} if $T_i(x)\neq T_j(x)$ whenever $i\neq j$ and $x\in C_i\cap C_j$.

In the case that the maps $T_i$ are invertible with $T_i(C_i)=X$, the collection of inverse mappings $T_i^{-1}$ is referred to as
an \emph{iterated function system} (IFS) on $X$, and as an \emph{overlapping IFS} if 
$C_i\cap C_j$ has non-empty interior for some $i\neq j$.
The study of the fractal geometry of the limit sets of such systems is an active area of research, 
 even in apparently simple settings (see e.g.~\cite{bakermemams, baker, bakerkoivusalo, bms, hochman, peressimonsolomyak, pollicottsimon, rams, 
sidorov, solomyak}).

The ergodic optimization of such systems is non-trivial irrespective of whether the limit set is well understood.
As a particular example, consider the case of $X=[0,1]$, $C_0=[0,1/2]$, $C_1=[1/4,3/4]$, $C_2=[1/2,1]$,
with $T_i:C_i\to[0,1]$ given by $T_i(x)=2x-i/2$ for $0\le i\le 2$.
We refer to the multi-valued map $T(x)=\{T_i(x):x\in C_i\}$ as the \emph{three-branch doubling map}, it being essentially an augmentation of the well known doubling map $D: x\mapsto 2x \pmod 1$ via the addition of a third branch
$T_1(x)=2x-1/2$; in particular, every $D$-invariant measure is also $T$-invariant.
Since the (single-valued) doubling map is much studied in ergodic optimization, a brief comparison with $T$ will be instructive. Specifically, the maximizing measures for the 1-parameter family of functions $f_\theta(x)=\cos 2\pi(x-\theta)$
are well understood, having been investigated (experimentally and conjecturally) by Hunt \& Ott \cite{HO96a, HO96b} 
and Jenkinson \cite{jenkinson1, jenkinson2}, before Bousch \cite{Bou00} rigorously proved that they are 
\emph{Sturmian}\footnote{In the specific context of the doubling map, an invariant probability measure is Sturmian if and only if its support is a subset of some closed semi-circle.
More generally,
the Sturmian measure of rotation number $\varrho$
can be defined on symbolic space as
the
push forward of Lebesgue measure on the circle
under the map sending points to their coding sequence under rotation by $\rho$.
Further details on Sturmian measures, sequences, and orbits,
can be found in e.g.~\cite{adjr, Bou00, bullettsentenac, jenkinson2, announce, major, morsehedlund}.}.
In particular, the maximizing measure is \emph{periodic} (i.e.~supported on a single periodic orbit) for typical parameter values $\theta$, but there exist infinitely many parameters (constituting a set that is both meagre, and of Lebesgue measure zero) for which the maximizing measure is non-periodic.

The effect of introducing the third branch $T_1$, thus allowing choice in the evolution of forward orbits, is to considerably increase the maximum ergodic average for functions in the family $f_\theta$, as illustrated in Figure 1.
Investigation of the fine structure of the maximizing measures for $(T,f_\theta)$ reveals that, as for 
 $(D,f_\theta)$, Sturmian measures are again present, though with a significant difference: whereas the support of a Sturmian measure for $D$ lies in a closed sub-interval of length $1/2$, for $T$ the support lies in a closed sub-interval of length $1/4$. The analogous phenomenon occurs for the one-parameter family of functions
$g_\theta(x)=-d(x,\theta)$, where $d$ denotes the Euclidean metric on $[0,1]$; the concavity of these functions means that the $(D,g_\theta)$-maximizing measure is Sturmian\footnote{A closely related setting is to consider the one-parameter of functions $h_\theta(x)=-\delta(x,\theta)$, where $\delta$ denotes the intrinsic metric 
$\delta([x],[y])=\min_{n\in\Z}|x-y-n|$
on the circle $\R/\Z$; in this case the $(D,h_\theta)$-maximizing measures are also known to be Sturmian \cite{adjr}.} by \cite{announce,major}, but (except for the case of the 
Dirac measure at the fixed point 0), such measures are not $(T,g_\theta)$-maximizing (rather, these latter measures are
not $D$-invariant, but are Sturmian for $T$ in the sense of having support contained in an interval of length $1/4$). 
A more complete exposition of this will appear elsewhere.

\begin{figure}[!h]
\begin{center}
 \includegraphics[scale=.9]{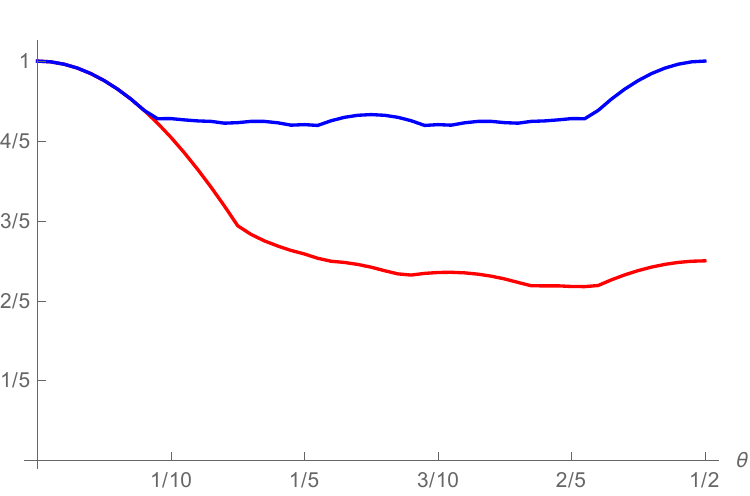}
\caption{Maximum ergodic averages $\beta(D,f_\theta)$ (red) and $\beta(T,f_\theta)$ (blue),
for $D$ the doubling map,
$T$ the three-branch doubling map, and $f_\theta(x)=\cos 2\pi(x-\theta)$, in the range $\theta\in[0,1/2]$. 
(Cf.~Figure 3 of \cite{HO96b}). For symmetry reasons the portion $\theta\in(1/2,1]$ can be omitted.}
\end{center}
\end{figure}

As a continuous map of the circle,
the doubling map can, after a change of coordinates, be written as a self-map $z\mapsto z^2$ of
$S^1=\{z\in\mathbb{C}:|z|=1\}$; in particular, it is a particular instance of a quadratic map 
$z\mapsto z^2+c$ restricted to its Julia set (cf.~e.g.~\cite[\S VIII]{cg}, \cite{devaney}).
    Siqueira \& Smania 
\cite{smaniasiqueira} and Siqueira \cite{siqueiraphd, siqueira1, siqueira2}
consider the multi-valued analogue, given by
 $z\mapsto z^{q/p} +c$;
    more precisely, for fixed $p ,\, q\in \N$ with $p<q$, 
and $c\in \mathbb{C}$,
 define
	\begin{equation*}
		F_{c,p,q} (z) := \bigl\{ w\in \overline{\mathbb{C}}  :  (w-c)^p =z^q \bigr\}
	\end{equation*}
    for all $z\in \overline{\mathbb{C}}$. The \emph{Julia set} $J(F_{c,p,q})$, defined as the closure of the 
set of repelling periodic points of $F_{c,p,q}$ (cf.~\cite[Defn.~6.31]{siqueiraphd}, \cite[\S~2.1]{smaniasiqueira}), 
is a closed subset of $\overline{\mathbb{C}}$. If $F_{c,p,q} |_J$ is the map given by $F_{c,p,q} |_J (z)
    := J(F_{c,p,q}) \cap F_{c,p,q} (z)$ for all $z\in J(F_{c,p,q})$,
    it follows from \cite[Thm.~2.7, Cor.~4.6, Cor.~4.8]{siqueiraphd} and \cite[Thms.~4.4, and~5.8]{siqueira1} that 
$(J(F_{c,p,q}), F_{c,p,q} |_J)$ is an open multi-valued topological dynamical system, that is expanding 
for certain parameters $c \in \mathbb{C}$, for example when $c$ is sufficiently
close to $0$.
            When $c=0$ and $p<q$, the Julia set is the unit circle,
and the system is equivalent to
$T(x) =\Bigl\{ \frac{qx}{p},\, \frac{qx+1}{p},\, \dots, \frac{qx+p-1}{p} \Bigr\}$
on $\R /\Z$, which is 
            open and expanding with respect to the intrinsic metric on $\R/ \Z$.
            
 The parameters $c=0$, $p=1$, $q=2$ give the doubling map $z\mapsto z^2$, and in this setting the 
set of \emph{barycentres} $\int_{S^1}z\, d\mu(z)$ of invariant measures is a certain compact convex set of the unit disc 
whose remarkable properties (in particular, the fact that its boundary contains a dense set of points of non-differentiability)
can be understood
(see \cite{Bou00, jenkinson1, jenkinson2,jenkinsonetds})
 in terms of the family $f_\theta(x)=\cos 2\pi(x-\theta)$; in particular, a measure has barycentre on the boundary if and only if it is Sturmian.
Calculations for other integers $p, q$ (with $c=0$ so that the Julia set is $S^1$) suggest that this may be a universal phenomenon (see Figure 2 for the case $p=2$, $q=3$), so we formulate the following Conjecture \ref{sturmian_conj},
and would expect that the Ma\~n\'e lemmas of Section \ref{mane_section} would underpin approaches to proving it,
in conjunction with specific estimates such as in \cite{Bou00}.

\begin{figure}[!h]
\begin{center}
 \includegraphics[scale=.99, trim=0cm 1cm 0cm 1cm, clip]{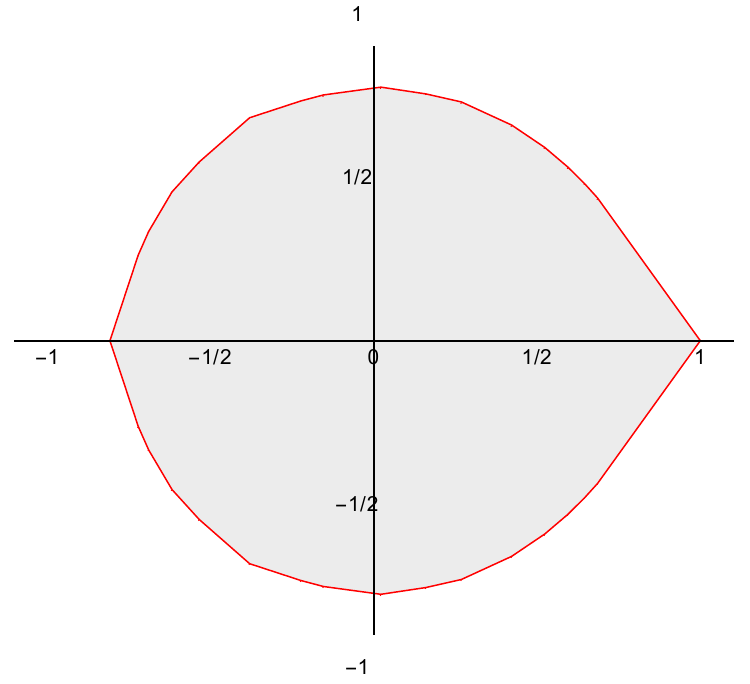}
\caption{Boundary of the set of barycentres $\int_{S^1}z\, d\mu(z)$ for probability measures $\mu$ invariant under 
$F_{0,2,3} (z) := \bigl\{ w\in \overline{\mathbb{C}}  :  w^2 =z^3 \bigr\}$. Extremal points in this figure
correspond to barycentres
of periodic orbits of period $\le 10$, all of which are Sturmian (i.e.~contained in some $1/3$-circle).}
\end{center}
\end{figure}

\begin{conj}\label{sturmian_conj} (Sturmian maximizing measures as extremal barycentres)

\noindent
For integers $1\le p<q$, the set of barycentres of probability measures on $S^1$ that are invariant
under $F_{0,p,q}$ has the property that a measure has barycentre on the boundary if and only if it is Sturmian (in the sense 
that its support is contained in some $1/q$-circle). 
\end{conj}

More generally, the specific open expanding systems we have considered suggest,  by analogy with the
conjectures\footnote{See also the later conjecture of Yuan \& Hunt \cite{YH99}, and the typical periodic optimization conjecture in \cite{new survey}.}
 of Hunt \& Ott,
and in view of the presence of a Ma\~n\'e lemma (see Section \ref{mane_section}),
 that there are no obvious obstacles to generalising the single-valued
\emph{topological} typical periodic optimization (TPO) 
of \cite{Co16} to the multi-valued setting, nor any reason to limit the
\emph{probabilistic}  typical periodic optimization (cf.~\cite{BZ16, DLZ24, GSZ25})
to the single-valued setting. We therefore articulate the following conjecture\footnote{Note that, without
the open expanding hypothesis, it is known that maximizing measures 
for multi-valued dynamical systems are \emph{typically unique} (see \cite{jllz} for 
further details, including the varying interpretations of \emph{typical} in that general context).}: 

\begin{conj}\label{TPOconj} (TPO for open expanding multi-valued dynamical systems)

\noindent Let $(X,T)$ be an open expanding multi-valued dynamical system, and let $\lip(X)$
denote the Banach space of Lipschitz real-valued functions on $X$. 
Let $P\subseteq \lip(X)$ consist of those functions whose maximizing measure is unique, and supported on a single periodic orbit.
Then 
$P$ contains an open dense subset of $\lip(X)$,
and is a prevalent\footnote{In the sense of \cite{huntsaueryorke} (cf.~also \cite{christensen2, huntsaueryorkeaddendum}).}
 subset of $\lip(X)$.
\end{conj}

\end{document}